\documentclass{amsart}

\usepackage[toc,page]{appendix}
\usepackage[headheight=12.1pt]{geometry}
\usepackage[all,cmtip]{xy}
\usepackage{amsmath}
\usepackage{amsfonts}
\usepackage{amssymb}
\usepackage{mathrsfs}   
\usepackage{amsthm}

\usepackage{xcolor}
\usepackage{cite}
\usepackage{stmaryrd}
\usepackage{graphicx}
\usepackage{pgf}
\usepackage{eepic}
\usepackage{pstricks}
\usepackage{epsfig}
\usepackage{fancyhdr}
\usepackage{tikz,caption,float}
\usepackage[backref=page]{hyperref}

\usepackage{euscript}

\usepackage{amsbsy}
\SetSymbolFont{stmry}{bold}{U}{stmry}{m}{n}
\usepackage{bm}

\usepackage{relsize}
\usepackage[bbgreekl]{mathbbol}
\DeclareSymbolFontAlphabet{\mathbb}{AMSb} 
\DeclareSymbolFontAlphabet{\mathbbl}{bbold}

\hypersetup{
     colorlinks   = true,
     citecolor    = red,
     linkcolor = blue,
     urlcolor = purple
}
    
  \newcommand\imCMsym[4][\mathord]{%
  \DeclareFontFamily{U} {#2}{}
  \DeclareFontShape{U}{#2}{m}{n}{
    <-6> #25
    <6-7> #26
    <7-8> #27
    <8-9> #28
    <9-10> #29
    <10-12> #210
    <12-> #212}{}
  \DeclareSymbolFont{CM#2} {U} {#2}{m}{n}
  \DeclareMathSymbol{#4}{#1}{CM#2}{#3}
}
\newcommand\alsoimCMsym[4][\mathord]{\DeclareMathSymbol{#4}{#1}{CM#2}{#3}}

\imCMsym{cmmi}{124}{\CMjmath}
\imCMsym[\mathop]{cmsy}{113}{\CMamalg}
\imCMsym[\mathop]{cmex}{96}{\CMcoprod}
\alsoimCMsym[\mathop]{cmex}{97}{\CMbigcoprod}
 
\numberwithin{equation}{subsection}

\let\realequation\equation
\def\equation{\setcounter{equation}{\arabic{subsubsection}}%
   \refstepcounter{subsubsection}%
   \realequation}
   
\makeatletter
\newcommand*{\relrelbarsep}{.386ex}
\newcommand*{\relrelbar}{%
  \mathrel{%
    \mathpalette\@relrelbar\relrelbarsep
  }%
}
\newcommand*{\@relrelbar}[2]{%
  \raise#2\hbox to 0pt{$\m@th#1\relbar$\hss}%
  \lower#2\hbox{$\m@th#1\relbar$}%
}
\providecommand*{\rightrightarrowsfill@}{%
  \arrowfill@\relrelbar\relrelbar\rightrightarrows
}
\providecommand*{\leftleftarrowsfill@}{%
  \arrowfill@\leftleftarrows\relrelbar\relrelbar
}
\providecommand*{\xrightrightarrows}[2][]{%
  \ext@arrow 0359\rightrightarrowsfill@{#1}{#2}%
}
\providecommand*{\xleftleftarrows}[2][]{%
  \ext@arrow 3095\leftleftarrowsfill@{#1}{#2}%
}
\makeatother

\usepackage{xr,mathtools}

\theoremstyle{plain}
\newtheorem*{theoremu}{Theorem}
\newtheorem{theorem}{Theorem}[section]

\newtheorem{proposition}[theorem]{Proposition}

\newtheorem{corollary}[theorem]{Corollary}

\newtheorem{lemma}[theorem]{Lemma}

\theoremstyle{definition}

\newtheorem{definition}[theorem]{Definition}

\newtheorem*{assumptionu}{Assumption}

\theoremstyle{remark}
\newtheorem{remark}[theorem]{Remark}

\newtheorem{example}[theorem]{Example}

\newtheorem*{claimu}{Claim}

\newcommand{\Z}{{\mathbb Z}}
\newcommand{\Q}{{\mathbb Q}}

\newcommand{\C}{{\mathbb C}}

\renewcommand{\P}{{\mathbb P}}
\newcommand{\A}{{\mathbb A}}
\newcommand{\D}{{\mathbb D}}

\newcommand{\ca}[1]{\mathcal{#1}}

\newcommand{\fr}[1]{\mathfrak{#1}}

\renewcommand{\lim}[1]{\underset{#1}{\mathrm{lim}}\,}

\newcommand{\isomto}{\overset{\cong}{\longrightarrow}}

\newcommand{\pow}[1]{\llbracket #1 \rrbracket}

\newcommand{\tate}[1]{\langle #1 \rangle}

\newcommand{\spec}[1]{\mathrm{Spec}\!\left(#1\right)}
\newcommand{\spa}[1]{\mathrm{Spa}\!\left(#1\right)}
\newcommand{\spf}[1]{\mathrm{Spf}\!\left(#1\right)}

\newcommand{\isoc}[1]{\mathrm{Isoc}^\dagger(#1)}

\newcommand{\rig}{\mathrm{rig}}
\newcommand{\dR}{\mathrm{dR}}

\newcommand{\et}{\mathrm{\acute{e}t}}

\title{A homotopy exact sequence for overconvergent isocrystals}

\begin{document}

\author{Christopher Lazda}
       \address[Lazda]{Warwick Mathematics Institute \\
       Zeeman Building \\
University of Warwick \\
Coventry \\CV4 7AL \\ UK}
       \email{chris.lazda@warwick.ac.uk}
       
\author{Ambrus P\'al}
       \address[P\'al]{Department of Mathematics\\
       Imperial College London \\
       Huxley Building, 180 Queen's Gate\\
       London, SW7 2AZ\\
       UK}
       \email{a.pal@imperial.ac.uk}       

\begin{abstract} In this article we prove exactness of the homotopy sequence of overconvergent fundamental groups for a smooth and projective morphism in characteristic $p$. We do so by first proving a corresponding result for rigid analytic varieties in characteristic $0$, following dos Santos \cite{dS15} in the algebraic case. In characteristic $p$, we then proceed by a series of reductions to the case of a liftable family of curves, where we can apply the rigid analytic result. We then use this to deduce a Lefschetz hyperplane theorem for convergent fundamental groups, as well as a comparison theorem with the \'etale fundamental group.
\end{abstract}

\maketitle

\tableofcontents

\section*{Introduction}

One of the basic principles in `algebraic' approaches to homotopy theory is that a smooth and proper morphism $f:X\rightarrow S$ of schemes (in any characteristic) should behave like a Serre fibration of topological spaces. In particular, for any reasonable definition of homotopy groups, one expects a long exact sequence relating the homotopy groups of the base $S$, the total space $X$, and the fibre $X_s$ over some point $s\in S$. For \'{e}tale homotopy groups, for example, this was proved in \cite{Fri73}, at least after completing away from the residue characteristics of $S$.

While the \'etale fundamental group controls the category of $\ell$-adic local systems on varieties in characteristics different from $\ell$, the same is certainly not true for $p$-adic local systems in characteristic $p$. In this situation, the version of the fundamental group that is usually considered is the one defined using Tannakian duality; this is somewhat analogous to the full pro-algebraic completion of the topological fundamental group $\pi_1(X)$ of a complex algebraic variety. (For now, we omit base points from the notation.) In this world of `pro-algebraic homotopy theory' much less is known than in \'etale homotopy theory, even in the case of smooth varieties over the complex numbers.

 For example, it is not completely clear what the correct analogues of the higher homotopy groups are (although see \cite{Toe00} for some work in this direction), and hence even formulating the analogue of the homotopy long exact sequence is problematic. Even if one sticks to the well-understood terms, i.e. to the sequence
 \[ \pi_1(X_s) \rightarrow \pi_1(X) \rightarrow \pi_1(S) \rightarrow \pi_0(X_s) \rightarrow \pi_0(X) \rightarrow \pi_0(S) \rightarrow * \]
 then showing exactness has in general proved to be rather difficult. If we work with de\thinspace Rham fundamental groups of complex varieties, and assume that the base $S$ and fibre $X_s$ are connected, then exactness follows from `right exactness of the pro-algebraic completion functor'. More generally it was shown for fields of characteristic $0$ using a mixture of algebraic and transcendental methods in \cite{Zha14}. Similar exact sequences for logarithmic and/or unipotent fundamental groups have been proved in \cite{CDPS,DPS18,Laz15}, and a version `over the generic point of $S$' in \cite{EH06}.
 
 A major new approach to these sorts of problems was introduced in \cite{dS15}, where the author showed how to construct push-forwards of certain kinds of `non-linear $\ca{D}$-modules', i.e. stratified schemes over the total space $X$. He then used this construction to give a completely algebraic proof of exactness of the $\pi_1$-part of the sequence, assuming geometric connectedness of  $X_s$ and $S$. One of the crucial insights of his article is that by replacing linear representations with projective representations, one can avoid completely one of the major difficulties in proving exactness of these sorts of `homotopy sequences' (see \S\ref{sec: exacri}).
 
 Inspired by dos Santos' methods, in this article we prove the following result.
 
 \begin{theoremu}[\ref{theo: main2}] Let $f:X\rightarrow S$ be a smooth, projective morphism of smooth varieties over a perfect field $k$ of characteristic $p>0$, with geometrically connected fibres and base. Let $s\in S(k)$. Then the sequence of fundamental groups
 \[ \pi_1^\dagger(X_s) \rightarrow \pi_1^\dagger(X) \rightarrow \pi_1^\dagger(S) \rightarrow 1 \]
 classifying overconvergent isocrystals is exact.
 \end{theoremu}
 
 If one tries to directly transport dos Santos' construction to the overconvergent setting, one is very quickly confronted by a seemingly insurmountable list of problems and subtleties: even in the linear case the problem of  constructing $\mathbf{R}^0f_*$ of an overconvergent isocrystal (without $F$-structure!) along a smooth and proper morphism is unreasonably difficult (see \S\ref{sec: adjoint}). Instead we proceed in a much more roundabout fashion, advancing via a lengthy chain of reductions, which here we present in reverse order to that found in the body of the article.
 
 First we cut our given morphism $f:X\rightarrow S$ by a sequence of hyperplane sections, which by some diagram chasing and a very weak form of the Lefschetz hyperplane theorem for fundamental groups allows us to reduce to the case of a family of curves. In this case, our morphism arises via pullback from the universal curve
 \[ \xymatrix{  X\ar[r]\ar[d] & \ca{C}_g \ar[d] \\ S \ar[r] & \ca{M}_g }  \]
 and hence by lifting the morphism $S\rightarrow \ca{M}_g$ along some smooth lift of $S$ (at least locally) we can assume that the whole family of smooth projective curves lifts to characteristic $0$. (This is not quite what we do, but this is the basic idea anyway.)
 
In this case we can write the overconvergent fundamental group $\pi_1^\dagger(S)$ as a quotient
\[ \varprojlim_\lambda \pi_1^\dR(V_\lambda) \twoheadrightarrow \pi_1^\dagger(S) \]
of the inverse limit of the de\thinspace Rham fundamental groups $\pi_1^\dR(V_\lambda)$ as $V_\lambda$ ranges over all strict neighbourhoods of the tube $]S[$ inside the generic fibre of the given lift. (Again, this is not strictly what we do but this is the essential idea.) Moreover, the same is true for $\pi_1^\dagger(X)$, and in fact we can choose a cofinal system $\{W_\lambda\}$ of strict neighbourhoods of $]X[$ such that each $W_\lambda$ is smooth and projective over $V_\lambda$. Some more diagram chasing then allows us to reduce to the following result in rigid analytic geometry.

 \begin{theoremu}[\ref{theo: main1}] Let $f:W\rightarrow V$ be a smooth, projective morphism of smooth analytic varieties over a $p$-adic field $K$, with geometrically connected fibres and base. Let $v\in V(K)$. Then the sequence of fundamental groups
 \[ \pi_1^\dR(W_v) \rightarrow \pi_1^\dR(W) \rightarrow \pi_1^\dR(V) \rightarrow 1 \]
 classifying coherent modules with integrable connection is exact.
 \end{theoremu}
 
 The point is that now we have reduced to a statement solely concerning smooth projective morphisms of analytic $K$-varieties, with no reference to tubes or overconvergence. We are therefore in a situation where we really can directly apply dos Santos' ideas and arguments, as essentially all of the difficulties we originally faced have disappeared. This is now what we do: the proof of this `de\thinspace Rham' homotopy exact sequence consists entirely of translating dos Santos' proof from \cite{dS15} into the analytic context. 
 
In actual fact, we do much less than this. Rather than reprove analytically all of dos Santos' results on `push-forwards' of stratified schemes, we instead use various tricks to be able to reduce to cases where we can in fact \emph{apply} his results. The basic idea is that it in fact suffices to show that for a stratified variety $Z$ over $W$, the unit map $f^*f_{\dR*} Z \rightarrow Z$ for the \emph{relative} push-forward is a closed immersion, and the image is stable under the stratification on $Z$. But now, by relative rigid analytic GAGA, these relative push-forwards simply arise as the analytification of those considered in \cite{dS15}. Moreover, that the image is stable under the stratification can be checked after passing to the completed local ring at any rigid point, and hence to the various infinitesimal neighbourhoods of this point. The situation is now completely algebraic over the ground field $K$, and so once more we can use dos Santos' results. In fact a little care is needed, since these infinitesimal neighbourhoods will not be smooth over $K$, so \emph{a priori} the results of \cite{dS15} do not apply. However, it is straightforward to show that the \emph{proofs} of these results apply in the situation we are interested in.

Finally, in \S\ref{sec: app}, we discuss some applications of the homotopy exact sequence for overconvergent fundamental groups. First of all we prove a Lefschetz hyperplane theorem, stating that if $Y\subset X$ is a smooth hyperplane section inside a smooth projective variety, then the induced map
\[ \pi_1^\dagger(Y)\rightarrow \pi_1^\dagger(X)  \]
on convergent fundamental groups is surjective. We are then able to use this to show that when the ground field is algebraically closed, and $X$ is smooth and projective, then there is a canonical isomorphism
\[ \pi_0(\pi_1^\dagger(X)) \cong \pi_1^\et(X)\]
between the component group of the convergent fundamental group and the pro-finite \'etale fundamental group. This generalises a result of Crew \cite[Proposition 4.4]{Cre92}, in which objects were assumed to have Frobenius structures.
  
\subsection*{Acknowledgements}

The first author was supported by a Marie Curie fellowship of the Istituto Nazionale di Alta Matematica ``F. Severi''. The second author was partially supported by the EPSRC Grant P36794. The authors would like to thank Imperial College London, Universit\`a Degli Studi di Padova, the Centre International de Rencontres Mathematiques in Marseille, and the Mittag-Leffler Institute in Stockholm for hospitality during the writing of this article. We would also like to thanks the anonymous referee for a careful reading of the article. Finally, we owe a huge amount to the work of J.P. dos Santos, in particular the article \cite{dS15}, which pointed the way towards a solution to a problem we had been considering for a number of years. 

\subsection*{Notations and conventions}

\begin{itemize}
\item We will denote by $k$ a perfect field of characteristic $p>0$, $\ca{V}$ a complete DVR with residue field $k$ and fraction field $K$ of characteristic $0$. We will let $\varpi$ denote a choice of uniformiser for $\ca{V}$.
\item An algebraic variety over $k$ (resp. $K$) will mean a separated scheme of finite type, the category of these will be denoted  $\mathbf{Var}_{k}$ (resp. $\mathbf{Var}_K$). If $V$ is an algebraic variety over either $k$ or $K$, we will denote by $\mathbf{Var}_V$ the slice category of varieties over $V$.
\item An analytic variety over $K$ will mean an adic space, separated and locally of finite type over $\spa{K,\ca{V}}$. Since all rigid spaces will be locally of finite type over $K$, we may, without ambiguity, denote an affinoid adic space $\spa{A,A^+}$ simply by $\spa{A}$. We will let $\mathbf{Rig}_{K}$ denote the category of analytic varieties over $K$, and for any such $V$ the slice category will be denoted $\mathbf{Rig}_V$. The analytification of an algebraic variety over $K$ will always be considered as an adic space. Throughout, we will implicitly use \cite[Theorem II.A.5.2]{FK18} to apply the results of \cite{FK18} to objects of $\mathbf{Rig}_K$.
\item If $Y$ is a $k$-variety, we will denote by $\mathrm{Isoc}^\dagger(Y/K)$ the category of overconvergent isocrystals on $Y/K$.
\item A closed subgroup of an affine group scheme will always mean a closed sub-scheme that is also a subgroup, a surjective homomorphism will be a group scheme homomorphism which is faithfully flat.
\item Unadorned tensor or fibre products will be over $k,K$ or $\ca{V}$, it will be clear which from the context. Sometimes, in order to avoid confusion, we will denote the fibre product of a diagram $X\overset{f}{\rightarrow} Z \overset{g}{\leftarrow} Y$ by one of $X\times_{f,Z,g} Y$, $X\times_{f,Z}Y$ or $X\times_{Z,g}Y$, depending on which structure morphism needs clarifying.
\end{itemize}

\section{The homotopy sequence for analytic \texorpdfstring{$K$}{K}-varieties}

The first goal of this article will be the proof of a homotopy exact sequence for certain classes of families of smooth analytic $K$-varieties. To start with, we will need to define the de\thinspace Rham fundamental group of such spaces. So let $V$ be an analytic variety over $K$. 

\begin{proposition} \label{prop: tann1} Assume that $V$ is smooth, geometrically connected, and admits a rational point $v\in V(K)$. Then the category $\mathrm{MIC}(V/K)$ of coherent $\ca{O}_V$-modules with integrable connection is neutral Tannakian over $K$, with fibre functor $v^*$.
\end{proposition}

\begin{proof}
We first claim that any coherent module with integrable connection is locally free. Indeed, this question is local, and we may assume $V=\spa{A}$ to be affinoid. In particular, $E$ comes from a coherent sheaf $E^a$ on $\spec{A}$ and it suffices to prove that $E^a$ is locally free. But this may be checked after passing to the completed local ring $\widehat{A}_\mathfrak{m}$ at any closed point $\mathfrak{m}\in \spec{A}$, which by enlarging $K$ can be assumed to be $K$-valued. Now choosing \'{e}tale co-ordinates $\spa{A}\rightarrow \D^n_K$ in some neighbourhood of this given $K$-point induces an isomorphism $\widehat{A}_\mathfrak{m}\cong K\pow{x_1,\ldots,x_n}$. Moreover the integrable connection on $E$ induces a formal integrable connection on $E^a\otimes_A \widehat{A}_\mathfrak{m}$. Hence we may apply \cite[Proposition 8.9]{Kat70}.

It therefore follows that
\[v^*: \mathrm{MIC}(V/K) \rightarrow \mathrm{Vec}_K  \]
is a faithful, $K$-linear, exact tensor functor, and since $V$ is connected, we can see that if $v^*(E)$ has dimension $1$, then $E$ is a line bundle. Hence applying \cite[Ch. II, Proposition 1.20]{DMOS82} it suffices to prove that the natural map
\[ K\rightarrow H^0_\mathrm{dR}(V/K):=H^0(V,\Omega^*_{V/K}) \]
is an isomorphism. Applying $v^*$ we obtain a retraction
\[ H^0_\mathrm{dR}(V/K)\cong \mathrm{End}_{\mathrm{MIC}(V/K)}(\ca{O}_V) \rightarrow \mathrm{End}_K(v^*\ca{O}_V)=K  \]
of this map. In particular, if $H^0_\mathrm{dR}(V/K)$ were strictly bigger than $K$, then $\Gamma(V,\ca{O}_V)$ would contain a non-trivial idempotent element, contradicting the connectedness of $V$.
\end{proof}

\begin{definition} Let $(V,v)$ be as in Proposition \ref{prop: tann1}. Then we define the de\thinspace Rham fundamental group $\pi_1^\mathrm{dR}(V,v)$ of $V$ to be the Tannaka dual of $\mathrm{MIC}(V/K)$ with respect to the fibre functor
\[ v^*: \mathrm{MIC}(V/K) \rightarrow \mathrm{Vec}_K . \]
\end{definition}

Now let $f:W\rightarrow V$ be a proper morphism of analytic $K$-varieties. Recall from \cite{Con06} that a line bundle $\ca{L}$ on $W$ is said to be $f$-ample if it is so on each fibre $W_v$ over a rigid point $v\in V$. In other words, for each rigid point $v\in V$, some tensor power $\ca{L}|_{W_v}^{\otimes{n}}$ defines a closed immersion $W_v\hookrightarrow \P^{N,\mathrm{an}}_{K(v)}$.

\begin{definition} We say that a proper morphism $f:W\rightarrow V$ of analytic $K$-varieties is projective if $W$ admits an $f$-ample line bundle.
\end{definition}

\begin{remark}
With this definition, a projective morphism admits a closed immersion $W\rightarrow \P^{N,\mathrm{an}}_V$ locally on the base $V$, by \cite[Theorem 3.2.7]{Con06}. Such an embedding need not exist globally, although it will if the base is affinoid, or itself projective over $\spa{K}$. Note also that with this definition, a composition of projective morphisms is projective, but projectivity is not necessarily local on the base.
\end{remark}

Suppose that $f:W\rightarrow V$ is a morphism of smooth, geometrically connected $K$-varieties, $w\in W(K)$ is a $K$-valued point, and set $v=f(w)$. If the fibre $W_v$ is also smooth and geometrically connected then we call the sequence
\[ \pi_1^\mathrm{dR}(W_v,w)\rightarrow \pi_1^\mathrm{dR}(W,w)\rightarrow \pi_1^\mathrm{dR}(V,v)\rightarrow 1  \]
of affine group schemes the \emph{homotopy sequence} associated to the pair $(f,w)$. Then the main result of the first part of this article is the following.

\begin{theorem} \label{theo: main1} Let $f:W\rightarrow V$ be a smooth projective morphism of smooth analytic $K$-varieties, with geometrically connected fibres and base, and let $w\in W(K)$. Then the homotopy sequence of the pair $(f,w)$ is exact.
\end{theorem}

\section{Exactness criteria and polarisable \texorpdfstring{$G$}{G}-varieties} \label{sec: exacri}

The strategy to prove Theorem \ref{theo: main1} is essentially to translate dos Santos' proof of exactness of the homotopy sequence in \cite{dS15} from the algebraic to the analytic setting. The need to work analytically will present us with several difficulties, and consequently at many points we will prove weaker results, and with extra hypotheses, than those obtained in \cite{dS15}. In order to be able to get away with this, we will need to combine the `projective' criteria for exactness of a sequence of affine group schemes discussed in \cite[\S4]{dS15} with more traditional `linear' versions considered for example in \cite[Appendix A]{EHS08}. To begin with then, let us quickly recall how these criteria work.

\begin{theorem}[\cite{dS15}, Lemma 4.3] \label{theo: exactds} Let
\[ L \overset{q}{\rightarrow} G \overset{p}{\rightarrow} A \rightarrow 1\]
be a sequence of affine group schemes such that $p$ is faithfully flat. Then the sequence is exact if and only if for all $V\in \mathrm{Rep}(G)$ the inclusion
\[ \P(V)^{\ker p}(K) \subset \P(V)^{L}(K) \]
of $K$-points on the fixed schemes is an equality.
\end{theorem}

\begin{theorem}[\cite{EHS08}, Theorem A.1(iii)] \label{theo: exactesnault} Let 
\[ 1\rightarrow L \overset{q}{\rightarrow} G \overset{p}{\rightarrow} A \rightarrow 1 \]
be a sequence of affine group schemes, such that $q$ is a closed immersion and $p$ is faithfully flat. Then the sequence is exact if and only if the following three conditions hold.
\begin{enumerate} 
\item If $V\in \mathrm{Rep}(G)$, then $q^*(V)$ is trivial in $\mathrm{Rep}(L)$ if and only if $V\cong p^*(W)$ for some $W\in \mathrm{Rep}(A)$;
\item for any $V\in \mathrm{Rep}(G)$, if $W_0\subset q^*(V)$ is the maximal trivial sub-object in $\mathrm{Rep}(L)$, then there exists $W\subset V \in \mathrm{Rep}(G)$ such that $q^*(W)=W_0\subset q^*(V)$;
\item any object of $\mathrm{Rep}(L)$ is a sub-object of one in the essential image of $q^*$.
\end{enumerate}
\end{theorem}

In practise, the first two of the conditions in Theorem \ref{theo: exactesnault} are (conceptually at least) very easy to verify, the third extremely difficult. It will therefore be useful to see what happens when we drop it. Note that the intersection of any collection of closed normal subgroups of an affine group scheme $G$ is also a closed normal subgroup, hence we may define the normal closure $H^\mathrm{norm}\subset G$ of a closed subgroup $H\subset G$ as the intersection of all closed normal subgroups containing it. 

\begin{definition} We say that a sequence of affine group schemes
\[ L \overset{q}{\rightarrow} G \overset{p}{\rightarrow} A \rightarrow 1 \]
is weakly exact if $G\overset{p}{\rightarrow} A$ is surjective, the composition $L\overset{p\circ q}{\rightarrow} A$ is trivial, and if $\ker(p)=q(L)^\mathrm{norm}$. In other words, the sequence
\[ 1\rightarrow q(L)^\mathrm{norm}\rightarrow G\rightarrow A \rightarrow 1\]
is exact.
\end{definition}

Weak exactness turns out to be exactly what we can prove without the third condition in Theorem \ref{theo: exactesnault}.

\begin{theorem} \label{theo: tann} Let $L\overset{q}{\rightarrow} G \overset{p}{\rightarrow} A \rightarrow 1 $ be a sequence of affine group schemes over $K$ such that $p$ is faithfully flat. Assume that:
\begin{enumerate} \item if $V\in \mathrm{Rep}(G)$, then $q^*(V)$ is trivial in $\mathrm{Rep}(L)$ if and only if $V\cong p^*(W)$ for some $W\in \mathrm{Rep}(A)$;
\item for any $V\in \mathrm{Rep}(G)$, if $W_0\subset q^*(V)$ is the maximal trivial sub-object in $\mathrm{Rep}(L)$, then there exists $W\subset V \in \mathrm{Rep}(G)$ such that $q^*(W)=W_0\subset q^*(V)$. 
\end{enumerate}
Then $L \rightarrow G \rightarrow A \rightarrow 1$ is weakly exact.
\end{theorem}

\begin{proof} First note that by \cite[Theorem A.1]{EHS08} we may describe $\mathrm{Rep}(q(L))$ as the full subcategory of $\mathrm{Rep}(L)$ consisting of objects which are sub-quotients of objects in the essential image of $q^*:\mathrm{Rep}(G)\rightarrow \mathrm{Rep}(L)$. In particular, it is straightforward to verify that both conditions continue to hold if we replace $L$ by $q(L)$, in other words we may assume that $q$ is a closed immersion and $L$ is in fact a closed subgroup of $G$.

We next claim that moreover the conditions continue to hold if we replace $L$ by the normal subgroup $L^\mathrm{norm}$ it generates, the non-trivial one is (2). In this case, we know from condition (2) applied to $L$ that for any representation $V$ of $G$, the subspace $V^L$ is in fact stable by $G$. Since $V^L$ is therefore a $G$-representation on which $L$ acts trivially, it follows that $L^\mathrm{norm}$ acts trivially, in particular we have $V^L=V^{L^\mathrm{norm}}$, which suffices to prove that (2) also holds for $L^\mathrm{norm}$. 

In other words we may in fact assume that $L=L^\mathrm{norm}$, and in particular that $L$ is a normal subgroup of $G$. But now we note that by \cite[Theorem A.1(ii)]{EHS08} any object of $\mathrm{Rep}(L)$ is a sub-object of one in the essential image of $q^*$, and hence applying Theorem \ref{theo: exactesnault} we can see that the sequence
\[ 1\rightarrow L \rightarrow G \rightarrow A\rightarrow 1 \]
is exact.
\end{proof}

As mentioned before, the conditions of Theorem \ref{theo: tann} are often easy to verify, and in the situation of Theorem \ref{theo: main1} we may do so as follows. Let $f:W\rightarrow V$, $w\in W(K)$ be as in Theorem \ref{theo: main1}, and suppose we are given $E\in \mathrm{MIC}(W/K)$. We define
\begin{align*}
f_{\mathrm{dR}*}E&:= \mathbf{R}^0f_*\left(E\otimes_{\ca{O}_W}\Omega^*_{W/V} \right) \\
&= \ker\left(f_*E \rightarrow f_*(E\otimes \Omega^1_{W/V}) \right) \\
&= f_*\ker\left(E\rightarrow E\otimes \Omega^1_{W/V} \right)
\end{align*}
to be the sheaf of relative horizontal sections. Since $f$ is proper, $f_{\mathrm{dR}*}E$ is a coherent sheaf on $V$, and exactly as in \cite[\S2]{KO68} we may endow it with an integrable connection. One easily verifies that 
\[ f^*:\mathrm{MIC}(V/K) \leftrightarrows \mathrm{MIC}(W/K) : f_{\mathrm{dR}*}\]
are adjoint functors, and that for any $K$-valued point $v\in V(K)$ there is a natural isomorphism
\[ v^*f_{\mathrm{dR}*}E \cong H^0_\mathrm{dR}(W_v/K,E|_{W_v}). \]

\begin{lemma}\label{lemma: dradj} In the situation of Theorem \ref{theo: main1} the sequence
\[ \pi_1^\mathrm{dR}(W_v,w)\rightarrow \pi_1^\mathrm{dR}(W,w)\rightarrow \pi_1^\mathrm{dR}(V,v) \rightarrow 1 \]
is weakly exact.
\end{lemma}

\begin{proof}
Using the fact that $v^*f_{\mathrm{dR}*}E \cong H^0_\mathrm{dR}(W_v/K,E|_{W_v})$ one easily checks that the adjunction map $F\rightarrow f_{\mathrm{dR}*}f^*F$ is an isomorphism for any $F\in \mathrm{MIC}(V/K)$, thus the functor $f^*$ is fully faithful. If we are given a sub-object $E\subset f^*F$, then again applying $f^*f_{\mathrm{dR}*}$ we obtain
\[ f^*f_{\mathrm{dR}*}E \subset E\subset f^*F \]
and we claim that in fact $ f^*f_{\mathrm{dR}*}E = E$. But since this can be checked on fibres, it follows from the fact that any sub-object of a trivial object in $\mathrm{MIC}(W_v/K)$ is itself trivial.

Hence the map $\pi_1^\mathrm{dR}(W,w)\rightarrow \pi_1^\mathrm{dR}(V,v)$ is faithfully flat. To show that condition (1) in Theorem \ref{theo: tann} holds, we note that for $E\in \mathrm{MIC}(W/K)$ the adjunction map $f^*f_{\dR*}E\rightarrow E$ is an isomorphism iff it is so on fibres, which happens iff $E|_{W_v}$ is trivial. Similarly, for (2) we can take $f^*f_{\dR*} E\subset E$ as the required sub-object.
\end{proof}

The reason that this is useful is that now we can formulate an alternative version of dos Santos' criterion (Theorem \ref{theo: exactds}).

\begin{proposition} \label{prop: mix} Let $L\overset{q}{\rightarrow} G \overset{p}{\rightarrow} A \rightarrow 1 $ be a weakly exact sequence of affine group schemes. Then the sequence is exact if and only if for any $V\in \mathrm{Rep}(G)$ the fixed scheme
\[ \P(V)^{L} \subset \P(V) \]
is invariant under $G$.
\end{proposition}

\begin{proof} By Theorem \ref{theo: exactds} we must prove that the inclusion $\P(V)^{\ker p}(K) \subset \P(V)^L(K)$ of $K$-points on the fixed scheme is an equality. If $\P(V)^L$ is invariant under $G$, then we obtain a homomorphism
\[ \rho: G\rightarrow \mathbf{Aut}_K(\P(V)^L) \]
of functors on $K$-schemes that by definition satisfies $q(L)\subset \ker \rho$. Since $\P(V)^L$ is a projective variety the functor $\mathbf{Aut}_K(\P(V)^L)$ is representable by a group scheme over $K$, hence $\ker  \rho$ is a \emph{closed} normal subgroup of $G$. Since it contains $q(L)$, it must also contain $q(L)^\mathrm{norm}$, from which we deduce that $\ker p$ must act trivially on $\P(V)^L$. Hence the claimed equality does indeed hold.
\end{proof}

This shows the importance of considering projective schemes together with actions of the fundamental group, and many results from \cite{dS15} involve extending the classical Tannakian duality to include these sorts of objects. We expect many of these results to also hold in the analytic context, but in our impatience to prove Theorem \ref{theo: main1} (and consequently Theorem \ref{theo: main2} below) we have not investigated this fully. Instead, we will stick to the more restrictive category of varieties together with a \emph{polarisable} action.

\begin{definition} \label{def: polar} Let $G$ be an affine group scheme over $K$, $Y$ a proper $K$-variety, and $\rho:G\rightarrow \mathbf{Aut}_K(Y)$ an action of $G$ on $Y$. We say that the action is polarisable if:
\begin{enumerate} \item $\rho$ factors through an algebraic quotient $G\twoheadrightarrow H$;
\item there exists an ample line bundle $\ca{L}$ on $Y$ admitting a $H$-linearisation.
\end{enumerate}
\end{definition}

Note that `polarisable' simply means that such an $H$ and $\ca{L}$ exist, we do not specify them as part of the data.

\begin{lemma} \label{lemma: polar} A $G$-action on $Y$ is polarisable if and only if there exists some $V\in \mathrm{Rep}(G)$ and a $G$-equivariant closed embedding
\[ Y \hookrightarrow \P(V).\]
\end{lemma}

\begin{proof} Since the action on any such $V$ must factor through an algebraic quotient, the existence of such an embedding clearly implies polarisability. For the converse, we may assume that $G$ is algebraic and that the line bundle $\ca{L}$ in condition (2) is very ample. In this situation, $H^0(Y,\ca{L})$ is a finite dimensional representation of $G$, and the natural map
\[ Y \rightarrow \P(H^0(Y,\ca{L}))\]
is $G$-equivariant.
\end{proof}

\begin{remark} In fact, the proof of this lemma shows that the condition in Definition \ref{def: polar} that the action of $G$ on either $Y$ or the ample line bundle $\ca{L}$ factors through some algebraic quotient is redundant.
\end{remark}

\begin{corollary} If $G\rightarrow H$ is a homomorphism of affine group schemes, and $Y$ is a proper $K$-variety with a polarisable $H$-action, then the induced $G$-action is also polarisable.
\end{corollary}

\section{Stratified analytic spaces}

The proof of Lemma \ref{lemma: dradj} demonstrates that the problem of proving weak exactness of the homotopy sequence is more or less that of constructing well-behaved `push-forwards' of coherent modules with integrable connections along the given map $f:W\rightarrow V$. Similarly, one of the key insights of \cite{dS15} is that the problem of proving that the conditions of Theorem \ref{theo: exactds} hold is essentially one of constructing push-forwards of more general, non-linear fibre bundles over $W$, endowed with `non-linear connections'. The construction of such push-forwards is exactly what we will want to imitate in the analytic setting. First, however, we will need to discuss the concept of a stratification on a analytic variety over some given base, which is the correct way to generalise integrable connections to non-linear objects.

So let $V/K$ be an analytic variety, which for now will we not necessarily assume to be smooth. Let $P_V^n$ denote the $n$th order infinitesimal neighbourhood of $V$ inside $V\times V$, and $p_i^n:P_V^n\rightarrow V$ for $i=0,1$ the projection maps.

\begin{definition} Let $Z\rightarrow V$ be an analytic variety over $V$. Then a stratification on $Z$ is a collection of compatible isomorphisms 
\[ \varepsilon_n:  Z\times_{p_0^n} P^n_V \isomto  P^n_V \times_{p_1^n} Z,\]
of $P^n_V$-varieties such that $\epsilon_0=\mathrm{id}$, and which satisfy the cocyle condition (see for example \cite[Ch. II, \S1]{Ber74}). A morphism of stratified varieties is simply a morphism compatible with the maps $\varepsilon_n$, and the category of such objects will be denoted $\mathbf{Str}(V/K)$. We will denote the full subcategory of $\mathbf{Str}(V/K)$ consisting of varieties $Z\rightarrow V$ which are \emph{projective} over $V$ by $\mathbf{StrP}(V/K)$. 
\end{definition}

\begin{example} \label{exa: strat} Assume that $V$ is smooth. \begin{enumerate}
\item If $Z\rightarrow V$ is a bundle of analytic affine spaces, in other words is locally isomorphic to the projection $\A^{n,\mathrm{an}}_V \rightarrow V$ (with linear transition maps), and the stratification maps
\[ Z\times_{p_0^n} P^n_V \isomto  P^n_V \times_{p_1^n} Z \]
are linear, then we recover the notion of a coherent module with integrable connection on $V$. 
\item \label{exa: strat2} If $E$ is a coherent $\ca{O}_V$-module with integrable connection, with associated affine bundle $\mathbf{E}\rightarrow V$, then the projectivisation $\P(E)\rightarrow V$ of $\mathbf{E}$ inherits a stratification from that on $E$.
\end{enumerate}
\end{example}

More generally, a natural source of stratified varieties will be varieties equipped with an action of the fundamental group. From now on we will assume that $V$ is smooth, connected, and equipped with a $K$-rational point $v\in V(K)$.

\begin{definition} We will denote the category of proper $K$-varieties together with a polarisable $\pi_1^\dR(V,v)$-action by $\ca{R}_V$.
\end{definition}

It is worth pointing out that the category we have denoted $\ca{R}_V$ is not the direct analogue in the analytic context of dos Santos' category of the same name considered in \cite[\S6.2]{dS15}. Our conditions are rather more restrictive, however, $\ca{R}_V$ will still have enough objects for our purposes.

The construction of objects in $\mathbf{StrP}(V/K)$ from those in $\ca{R}_V$ is relatively straightforward. Indeed, if $(Y,\rho) \in \ca{R}_V$ then we may choose an equivariant embedding $Y\hookrightarrow \P^N_K$ for some linear action of $\pi_1^\dR(V,v)$ on $\A^{N+1}_K$. Via this embedding we may view the projective co-ordinate ring
\[ S_Y := \bigoplus_n \Gamma(Y,\ca{O}_{\P^N_K}(n) ) \]
as a $\pi_1^\dR(V,v)$-representation. By construction, we know that $S_Y$ is the colimit of its finite dimensional sub-representations, and hence via the usual Tannakian correspondence we can construct an associated ind-coherent sheaf $\ca{S}_Y$ of graded rings on $V$, equipped with an integrable connection. We now define
\[ \ca{U}_V(Y,\rho):= \mathbf{Proj}_{\ca{O}_V} ( \ca{S}_Y )\]
via the relative Proj construction of \cite{Con06}. The integrable connection on $\ca{S}_Y$ induces a stratification on $\ca{U}_V(Y,\rho)$, making it into an object of $\mathbf{StrP}(V/K)$. This generalises Example \ref{exa: strat}(\ref{exa: strat2}) in that if $Y=\P(E_v)$ for some $E\in \mathrm{MIC}(V/K)$, then $\ca{U}_V(Y)\cong \P(E)$.

\begin{proposition} \label{prop: functor} This construction induces a functor $\ca{U}_V:\ca{R}_V\rightarrow \mathbf{StrP}(V/K)$ from polarisable $\pi_1^\dR(V,v)$-varieties to projective stratified $V$-varieties. It is compatible with pull-back via morphisms $f:W\rightarrow V$ in the sense that the diagram
\[  \xymatrix{ \ca{R}_V \ar[d]\ar[r] & \mathbf{StrP}(V/K)\ar[d] \\ \ca{R}_W\ar[r] & \mathbf{StrP}(W/K)  } \]
is 2-commutative.
\end{proposition}

\begin{proof} There are two things to check: firstly that $\ca{U}_V(Y,\rho)$ does not depend on the choice of $\pi_1^\dR(V,v)$-linearised ample line bundle $\ca{L}$, and secondly that we can make the association functorial in $(Y,\rho)$. For the first claim, we note that given two equivariant embeddings $Y\rightarrow \P(E_v)$ and $Y\rightarrow \P(F_v)$, giving rise to polarised $\pi_1^\dR(V,v)$-varieties $(Y,\rho)$ and $(Y,\sigma)$, we can simply consider their product
\[ Y\hookrightarrow \P(E_v) \times_K \P(F_v)  \]
as a polarised $\pi_1^\dR(V,v)$-variety $(Y,\tau)$, and show that the two projection maps induce isomorphisms
\[ \ca{U}_V(Y,\tau)\isomto \ca{U}_V(Y,\rho),\;\;\;\;\ca{U}_V(Y,\tau)\isomto \ca{U}_V(Y,\sigma) \]
Similarly, to obtain functoriality, we factor a given map $Y\rightarrow Y'$ of polarisable $\pi_1^\dR(V,v)$-varieties into the closed immersion $Y\hookrightarrow Y\times_K Y'$ followed by the projection $Y\times_K Y' \rightarrow Y'$ to reduce to considering either the case of a closed immersion or that of a projection. These can both be very easily handled.

Finally, functoriality in $f:W\rightarrow V$ follows from the facts that the homomorphism $\pi_1^{\dR}(W,w)\rightarrow \pi_1^\dR(V,v)$ corresponds to $f^*$ on the level of modules with integrable connection, and that $\mathbf{Proj}$ commutes with pull-back of ind-coherent modules by \cite[Theorem 2.3.6]{Con06}.
\end{proof}

With additional polarisability assumptions, `Tannakian reconstruction' theorems are very easy to prove using the classical linear versions.

\begin{definition} We say that $Z\in \mathbf{StrP}(V/K)$ is polarisable if there exists $E\in \mathrm{MIC}(V/K)$ and a closed embedding $Z\hookrightarrow \P(E)$ of stratified $V$-varieties.
\end{definition}

Clearly, the functor $\ca{R}_V$ lands inside the full subcategory $\mathbf{StrPol}(V/K)\subset \mathbf{StrP}(V/K)$ consisting of polarisable stratified $V$-varieties. For us, polarisable stratified varieties will always assumed to be projective, although this is not clear from the terminology. 

\begin{theorem}  \label{theo: tanninv}  The functor
\[ \ca{U}_V: \ca{R}_V\rightarrow \mathbf{StrPol}(V/K) \]
is an equivalence of categories.
\end{theorem}

\begin{proof} This follows very easily from ordinary Tannakian duality and the relative Proj construction introduced in \cite{Con06}. We first claim that given $(Y,\rho)\in \ca{R}_V$, the functor $\ca{U}_V$ induces a bijection between $\pi_1^{\dR}(V,v)$-invariant sub-schemes of $Y$ and closed stratified sub-varieties of $\ca{U}_V(Y,\rho)$. Indeed, injectivity is clear, since for $T\subset Y$ closed and $\pi_1^\dR(V,v)$-invariant we may recover $T$ as the fibre of $\ca{U}_V(T,\rho)$ over $v$.

For surjectivity, we note that by construction $\ca{U}_V(Y,\rho)=\mathbf{Proj}_{\ca{O}_V}(\ca{S}_Y)$ for some ind-coherent graded $\ca{O}_V$-algebra $\ca{S}_Y$, equipped with an integrable connection. The closed sub-variety $\ca{T}\hookrightarrow \ca{U}_V(Y,\rho)$ is therefore given by some quotient
\[ \ca{S}_Y \rightarrow \ca{S}_\ca{T} \]
which, since $T$ is a stratified sub-variety, must be horizontal. There is therefore an induced integrable connection on $\ca{S}_{\ca{T}}$. Moreover, since $\ca{S}_Y$ is the colimit of its coherent, horizontal sub-bundles, the  same is true of $\ca{S}_\ca{T}$. Hence by the usual Tannakian correspondence this has to come from some $\pi_1^{\dR}(V,v)$-invariant quotient $S_Y:=\ca{S}_{Y,v}\rightarrow S_T:=\ca{S}_{\ca{T},v}$. Then $T=\mathrm{Proj}(S_T)$ is the required invariant closed sub-scheme of $Y$.

This immediately implies essential surjectivity of $\ca{U}_V$, and in fact also implies full faithfulness. Indeed, as in Proposition \ref{prop: functor}, to prove full faithfulness it suffices via the graph construction to treat closed immersions and projections from products. The latter is obvious, and we have just proved the former. \end{proof}

\section{Relative stratifications and push-forwards} \label{sec: push}

To construct appropriate `push-forwards' of smooth, projective stratified varieties along a smooth and projective morphism in \cite{dS15}, dos Santos proceeds in two stages. First of all he considers the push-forward of a `relatively stratified variety' $Z$, and then shows that when this arises from a variety with an `absolute' stratification, there is a canonical induced stratification on this push-forward. The analogy to bear in mind from the linear case is that the push-forward $f_{\dR*}E$ of some module with integrable connection $E\in \mathrm{MIC}(W/K)$ is constructed by first viewing it as an object in $\mathrm{MIC}(W/V)$; one then puts a connection on $f_{\dR*}E$ by using the fact that $E$ came from $\mathrm{MIC}(W/K)$.

In this section we will achieve the first step for stratified analytic varieties by appealing to GAGA. This will tell us that we can actually apply dos Santos' results to provide the required push-forwards. In \S\ref{sec: pf1} below, we will then find another way to complete the proof of Theorem \ref{theo: main1} \emph{without} having to develop the analytic analogue of the `infinitesimal equivalence relations' used in \cite{dS15}, instead by reducing to the situation over the formal polydisc $\spf{K\pow{x_1,\ldots,x_n}}$. We start by introducing certain `formal adic spaces', which will allow a slightly better way of talking about stratifications.

\begin{definition} Let $T\hookrightarrow V$ be a closed immersion of analytic $K$-varieties, it is therefore by \cite[Proposition II.7.3.5]{FK18} defined by a coherent ideal sheaf $\ca{I}_T\subset \ca{O}_V$. Let $V^{(n)}_T$ denote the closed sub-variety of $V$ defined by the ideal sheaf $\ca{I}_T^{n+1}$ (i.e. $V_T^{(n)}$ is the $n$th infinitesimal neighbourhood of $T$ in $V$). We define the `formal completion of $V$ along $T$' to be the ind-object $V_{/T}:=\{ V^{(n)}_T \}_n\in \mathrm{Ind}(\mathbf{Rig}_K)$ in the category of analytic $K$-varieties.
\end{definition}

Let $\mathrm{Sh}(\mathbf{Rig}_K)$ denote the category of sheaves on $\mathbf{Rig}_K$ for the analytic topology. Since objects of $\mathbf{Rig}_K$ are locally quasi-compact, we have a fully faithful embedding
\[\mathrm{Ind}(\mathbf{Rig}_K) \hookrightarrow \mathrm{Sh}(\mathbf{Rig}_K) \]
and we will use this to view $V_{/T}$ as such a sheaf.

\begin{example} If $V=\spa{K\tate{x},\ca{V}\tate{x}}$ and $T\hookrightarrow V$ is the zero section, then
\[ V^{(n)}_T = \spa{\frac{K[x]}{(x^{n+1})},\ca{V}+x\frac{K[x]}{(x^{n+1})}}. \]
So we should think of $V_{/T}$ as being given by something like
\[  ``\spa{K\pow{x},\ca{V}+xK\pow{x}}"  \]
where the topology on $K\pow{x}$ has a basis of open subgroups of the form $\varpi^m\ca{V}\pow{x}+x^nK\pow{x}$. Note that with this topology, $K\pow{x}$ is not an $f$-adic ring, and hence the pair $(K\pow{x},\ca{V}+xK\pow{x})$ is \emph{not} an affinoid ring in the sense of \cite[\S1.1]{Hub96}. It would be interesting to see if there is a more general category of adic spaces in which things like $\spa{K\pow{x},\ca{V}+xK\pow{x}}$ make sense.
\end{example}

By considering the diagonal $\Delta:V\rightarrow V^2$ of an analytic $K$-variety, we obtain the ind-variety that we will denote by $\widehat{P}_V$, which comes equipped with two `projection' maps $p_i:\widehat{P}_V\rightarrow V$. With this language, we can rephrase the data of a stratification on some variety $Z \rightarrow V$ as an isomorphism
\[\varepsilon:  Z \times_{p_0} \widehat{P}_V  \isomto  \widehat{P}_V \times_{p_1} Z \]
in the slice category $\mathrm{Sh}(\mathbf{Rig}_K)_{/\widehat{P}_Z}$ of sheaves over $\widehat{P}_V$, subject to certain obvious conditions. If we let $s: \widehat{P}_V\rightarrow \widehat{P}_V$ denote the map switching the factors and $c :\widehat{P}_V\times_{p_1,V,p_0} \widehat{P}_V \rightarrow \widehat{P}_V$ the map induced by $((v_0,v_1),(v_1,v_2))\mapsto (v_0,v_2)$, then exactly as in \cite[Ch. II, \S1]{Ber74} we can show that the data
\[ p_0,p_1:\widehat{P}_V \rightrightarrows V,\;\;  c :\widehat{P}_V\times_{p_1,V,p_0} \widehat{P}_V \rightarrow \widehat{P}_V ,\;\;\Delta: V\rightarrow\widehat{P}_V,\;\;\ s:\widehat{P}_V\rightarrow \widehat{P}_V\]
forms a `formal groupoid' over $\spa{K}$, and that a stratification on a $V$-variety $Z$ is equivalent to an action of this groupoid.

Similarly, if we are given some morphism $f:W\rightarrow V$, then we may consider the formal completion $\widehat{P}_{W/V}$ along the diagonal $\Delta:W\rightarrow W\times_V W$. We have
\[ p_0,p_1:\widehat{P}_{W/V} \rightrightarrows W,\;\;  c :\widehat{P}_{W/V}\times_{p_1,W,p_0} \widehat{P}_{W/V} \rightarrow \widehat{P}_{W/V} ,\;\;\Delta: W\rightarrow\widehat{P}_{W/V},\;\;\ s:\widehat{P}_{W/V}\rightarrow \widehat{P}_{W/V}\]
exactly as before, giving rise to a formal groupoid over $V$.

\begin{definition} A $V$-linear stratification on a $W$-variety $Z$ is an action of the groupoid $\widehat{P}_{W/V}\rightrightarrows W$. We denote the category of $W$-varieties with a $V$-linear stratification by $\mathbf{Str}(W/V)$, and the full subcategory of objects which are projective over $W$ by $\mathbf{StrP}(W/V)$
\end{definition} 

These notions satisfy all the usual functorialities, which can be summarised by say that for any commutative square
\[ \xymatrix{ W'\ar[r]\ar[d] & W \ar[d] \\ V'\ar[r] & V  } \]
there is an induced morphism
\[ \left[ \widehat{P}_{W'/V'}\rightrightarrows W' \right] \rightarrow \left[ \widehat{P}_{W/V}\rightrightarrows W \right] \]
of formal groupoids, and hence a pull-back functor $\mathbf{Str}(W/V)\rightarrow \mathbf{Str}(W'/V')$. This is transitive in the obvious manner. For example taking the diagram
\[ \xymatrix{ W\ar[r]\ar[d] & W \ar[d] \\ V\ar[r] & \spa{K} } \]
we obtain the forgetful functor $\mathbf{Str}(W/K)\rightarrow \mathbf{Str}(W/V)$. 

Now let us suppose that we have a smooth, projective morphism $f:W\rightarrow V$ of analytic $K$-varieties, with geometrically connected fibres. Note that we \emph{do not} assume at this point that the base $V$  is smooth. If $\mathbf{P}(V)=\mathbf{StrP}(V/V)$ denotes the category of projective $V$-varieties, then as we have just seen there is a pull-back functor
\[f^*:\mathbf{P}(V) \rightarrow \mathbf{StrP}(W/V). \]
We wish to construct an `adjoint' to $f^*$, at least under certain extra conditions on objects in $\mathbf{StrP}(W/V)$. To do so, suppose therefore that we are given some $Z\in \mathbf{StrP}(W/V)$. Define a functor
\begin{align*} f_*Z :\mathbf{Rig}_V &\rightarrow \mathbf{Sets} \\
T/V &\mapsto \left\{  \text{sections of } Z \times_V T \rightarrow W\times_V T  \right\}
\end{align*}
where sections are considered as certain closed sub-varieties of  $Z \times_V T$. (Thus $f_*Z$ is a sub-functor of the Hilbert functor, the usual flatness condition is redundant for $f_*Z$, since $W$ is flat over $V$.)

\begin{proposition}  \label{prop: rep1} The functor $f_*Z$ is representable by an analytic variety over $V$ which has the following property: for each open affinoid $\spa{A}\subset V$, the restriction of $f_*Z \times_V \spa{A}\rightarrow \spa{A}$ to each of its connected components arises as the analytification of a quasi-projective $A$-scheme.
\end{proposition}

\begin{proof} This is similar in spirit to \cite[Theorem 4.1.3]{Con06}. Since $f_*Z$ is clearly a sheaf for the analytic topology on $\mathbf{Rig}_K$ we may in fact assume that $V=\spa{A}$ is affinoid. Hence, by relative rigid analytic GAGA \cite[Example 3.2.6]{Con06}, $W$ is the analytification of a smooth projective $A$-scheme $W^a$, and $Z\rightarrow W$ is the analyitification of a projective morphism $Z^a\rightarrow W^a$. We consider the corresponding functor
\begin{align*}
f^a_*Z^a :\mathbf{Sch}_A &\rightarrow \mathbf{Sets} \\
T/A &\mapsto \left\{  \text{sections of } Z^a \times_{A} T \rightarrow W^a \times_A T  \right\}
\end{align*}
of locally Noetherian $A$-schemes, which by \cite[\S4, Variant c.]{Gro60} is representable by a disjoint union of quasi-projective $A$-schemes. It therefore suffices to show that the analytification of $f^a_*Z^a$ represents the functor $f_*Z$. Since both are sheaves for the analytic topology, it suffices to check this on affinoids $\spa{B}\rightarrow \spa{A}$. In this case, we can again appeal to rigid analytic GAGA, which says that any closed sub-variety of $Z\times_{\spa{A}} \spa{B}$ is algebraic, i.e. comes from a unique closed sub-scheme of $Z^a_{B}$.
\end{proof}

Let $b:T\rightarrow f_*Z$ be a point of $f_*Z$, corresponding to a section $\tau_b:W\times_V T \rightarrow Z\times_V T$. Pulling back by the two projections $p_i:\widehat{P}_{W/V}\rightarrow W$, i.e. applying  $\widehat{P}_{W/V}\times_{p_i,W}$, we obtain sections 
\[p_i^*(\tau_b): \widehat{P}_{W/V} \times_V T \rightarrow \widehat{P}_{W/V}\times_{p_i,W}Z\times_V T. \]
of $\mathrm{id}\times_{p_i,W}g$. We say that $b$ is \emph{horizontal} if $\epsilon(p^*_0(\tau_b))=p^*_1(\tau_b)$, where $\epsilon$ is the stratification on $Z$.

\begin{definition} We define $f_{\dR*}Z\subset f_*Z$ to be the sub-functor of horizontal sections.
\end{definition}

\begin{proposition} \label{prop: rep2} The sub-functor $f_{\dR*}Z$ is representable by a closed analytic sub-variety of $f_*Z$. If $Z$ is smooth over $W$, then for any open affinoid $\spa{A}\subset V$ the restriction of $f_{\dR*}Z \times_V \spa{A}\rightarrow \spa{A}$ to each of its connected components is projective. 
\end{proposition}

\begin{proof} Being a closed sub-variety is local, and hence we may in fact assume that $V=\spa{A}$ is affinoid. Now again the whole situation algebrises: we have some smooth projective $f^a:W^a\rightarrow \spec{A}$ and some projective $Z^a\rightarrow W^a$ giving rise to $Z$ upon analytification. Moreover, since the algebraic infinitesimal neighbourhoods give rise to the analytic ones upon analytification, it follows that the analytic stratification on $Z$ comes from a unique  $A$-linear algebraic stratification on $Z^a$. Now we simply note that the results of \cite[\S10]{dS15} apply over any separated, Noetherian base scheme, for example, $\spec{A}$. Translated into algebraic terms, what we have termed `horizontal' corresponds exactly to what dos Santos calls `tangential', hence we may again use rigid analytic GAGA to show that our functor $ f_{\dR*}Z$ is simply the analytification of dos Santos' scheme $H_{f^a}(Z^a)$.
\end{proof}

The defining property of $f_{\dR*}Z$ gives a section $W\times_V f_{\dR*}Z \rightarrow Z\times_V f_{\dR*}Z$, and by composing with the first projection we therefore obtain a morphism $\varepsilon_Z: f^*f_{\dR*}Z\rightarrow Z$ of $W$-varieties. Essentially all of the main properties of $f_{\dR*}Z$ can then be deduced from those proved in \cite{dS15}.

\begin{proposition} \label{prop: props} Let $Z\in \mathbf{StrP}(W/V)$ be smooth over $W$. \begin{enumerate}
\item The map $\varepsilon_Z: f^*f_{\dR*}Z\rightarrow Z$ is horizontal with respect to the pull-back ($V$-linear) stratification on $f_{\dR*}Z$ and the given ($V$-linear)  stratification on $Z$.
\item \label{prop: pp2} Formation of $f_{\dR*}Z$ is compatible with base change: if $V'\rightarrow V$ is a morphism of $K$-varieties, then $(f_{\dR*}Z)\times_V V' \cong f_{\dR*}(Z\times_V V')$, and, via this isomorphism, $\varepsilon_{Z\times_V V'} = \varepsilon_Z \times_V V'$.
\item \label{prop: pp3} If the base $V=\spa{K}$ is a point, $w\in W(K)$ is a rational point, and $Z=\ca{U}_W(Y,\rho)$ for some $(Y,\rho)\in \ca{R}_W$, smooth over $K$, then $\varepsilon_Z: f^*f_{\dR*}Z \rightarrow Z$ is obtained by applying $\ca{U}_W$ to the closed immersion
\[ Y^{\pi_1^\dR(W,w)}\rightarrow Y, \]
considered as a morphism of $\pi_1^\dR(W,w)$-varieties.
\end{enumerate}
\end{proposition}

\begin{proof} Note that the first two are local on $V$, and therefore follow from their algebraic versions \cite[Proposition 12.1, Corollary 12.4]{dS15}. For the third, we wish to algebrise and apply \cite[Proposition 14.5]{dS15}, the point is to check that the induced (algebraic) stratification on $Z$ is simple. This follows from the argument in the paragraph preceding the proof of \cite[Proposition 14.5]{dS15}.
\end{proof}

\section{Formal stratifications, integrability and base change}

It will be necessary for us to have a formal analogue of the above constructions, working for now over an arbitrary field $F$. Since the basic ideas are essentially identical to those in the previous section, we will not give too many details. Let $t=(t_1,\ldots,t_d)$ be a collection of variables and set $S=\spf{F\pow{t}}$ to be the $d$-dimensional formal polydisc over $F$ (using the $t$-adic topology). We will let $f:X\rightarrow S$ be a smooth morphism of finite type. In this situation we may define the formal groupoids 
\[\widehat{P}_X \rightrightarrows X,\;\; \widehat{P}_{X/S} \rightrightarrows X,\;\; \widehat{P}_S \rightrightarrows S \]
\emph{exactly} as before. If $g:Z\rightarrow X$ is a morphism of formal schemes, then an $F$-linear (resp $S$-linear) stratification on $Z$ is then an action of the groupoid $\widehat{P}_X \rightrightarrows X$ (resp. $\widehat{P}_{X/S} \rightrightarrows X$). We similarly have the notion of an $F$-linear stratification on some formal $S$-scheme.

If $X/S$ is projective, and $g:Z\rightarrow X$ is a projective formal $X$-scheme equipped with an $S$-linear stratification, then we may define the push-forward $f_*Z$ as a disjoint union of formal schemes over $S$, as well as the closed sub-scheme $f_{\dR*}Z\subset f_{*}Z$, exactly as in the previous section. If we let $S_n,X_n,Z_n$ denote the mod $t^{n+1}$-reductions, then we could equally well construct $f_{\dR*}Z$ and $f_{*}Z$ as the limits
\[ f_{\dR*}Z= \mathrm{colim}_n ( f_{n,\dR*}Z_n) ,\;\;  f_{*}Z=\mathrm{colim}_n ( f_{n,*}Z_n ) \]
of the algebraic push-forwards along $f_n:X_n\rightarrow S_n$, as considered by dos Santos in \cite[\S10]{dS15} (and that he terms $H_0(Z_n)$ and $H_{f_n}(Z_n)$ respectively). The following represents a simple extension to formal schemes of the results of \cite{dS15}.

\begin{theorem} \label{theo: stratfor} Assume that $X$ is projective over $S$, and that $Z$ is smooth and projective over $X$, equipped with an $F$-linear stratification. Then there exists an $F$-linear stratification on $f_{\dR*}Z$ as a formal $S$-scheme such that the map
\[ f^*f_{\dR*}Z \rightarrow Z \]
is compatible with the $F$-linear stratifications on both sides.
\end{theorem}

\begin{proof} This essentially follows from \cite{dS15} upon taking the limit in $n$, although a little care is needed to achieve this. Consider for all $n$ the mod $t^{n+1}$ reduction $Z_n \overset{g_n}{\rightarrow} X_n \overset{f_n}{\rightarrow} S_n$ of everything in sight, we wish to construct a stratification on $f_{n,\dR*}Z_n$ as an $S_n$-scheme such that $f_n^*f_{n,\dR*}Z_n \rightarrow Z_n$ is compatible with the stratifications. We cannot directly apply the results of \cite{dS15} since $S_n$ is not smooth over $F$, but we can get around this as follows.

Firstly, let us recall that dos Santos views stratifications as particular kinds of `infinitesimal equivalence relations', and \emph{without} using any smoothness assumptions on the `base' of the fibration (in our case $S_n$) he constructs an infinitesimal equivalence relation on $f_{n,\dR*}Z_n$ such that $f_n^*f_{n,\dR*}Z_n \rightarrow Z_n$ intertwines these two equivalence relations. The point is then to try to prove that this equivalence relation on $f_{n,\dR*}Z_n$ actually comes from a stratification. This is the content of \cite[Proposition 13.4]{dS15}, and this proposition is the \emph{only} place where smoothness assumptions are used.

But here we can exploit the fact that our situation arises as the mod $t^{n+1}$ reduction of $Z\overset{g}{\rightarrow} X \overset{f}{\rightarrow} S$ with $S$ formally smooth over $K$. In particular, if we choose local \'etale co-ordinates $x=(x_1,\ldots, x_k)$ for $X/S$ and $z=(z_1,\ldots,z_m)$ for $Z/X$, then the stratification on $Z$ corresponds to some section $\ca{O}_Z\pow{dt,dx,dz} \rightarrow \ca{O}_Z\pow{dt,dx}$ of the natural inclusion $\ca{O}_Z\pow{dt,dx} \rightarrow \ca{O}_Z\pow{dt,dx,dz}$. In particular we may therefore choose $m$ elements $F_1,\ldots,F_m$ generating the kernel, locally on $X$ and $Z$. Now reducing mod $t^{n+1}$ we can follow the proof of \cite[Proposition 13.4]{dS15} word for word to conclude.
\end{proof}

For this to be useful, we will need to compare this with the set-up considered previously, let us therefore return to that situation. So we have some smooth projective map $f:W\rightarrow V$ of analytic $K$-varieties, with geometrically connected fibres, and $g:Z\rightarrow W$ some smooth and projective stratified variety over $W$. If $v\in V$ is a \emph{smooth} rigid point of $V$, then we may consider the various infinitesimal neighbourhoods $v^{(n)}:=V_{v}^{(n)}$ as before. The point is that now the base change $Z_{v^{(n)}}\rightarrow W_{v^{(n)}} \rightarrow v^{(n)}$ algebrises relative to the ground field $K$, so we may consider $Z_{v^{(n)}}\rightarrow W_{v^{(n)}}$ as a projective morphism of projective $K$-schemes, equipped with a $K$-linear stratification. Thus taking the limit in $n$ we obtain a smooth projective stratified formal scheme $\widehat{Z}_v\overset{g_v}{\rightarrow} \widehat{W}_v \overset{f_v}{\rightarrow} \mathrm{Spf}(\widehat{\ca{O}}_{V,v})$, where $\widehat{\ca{O}}_{V,v}$ is simply considered with the maximal-adic topology (not any kind of $p$-adic topology). Since $v$ was chosen to be a smooth point, we therefore find ourselves in the situation of Theorem \ref{theo: stratfor}. 

Hence we have some projective stratified formal scheme $f_{v,\dR*}\widehat{Z}_v$ over $\mathrm{Spf}(\widehat{\ca{O}}_{V,v})$. Alternatively, since the base change of $f_{\dR*}Z$ to $v^{(n)}$ is a disjoint union of projective $v^{(n)}$-schemes, we may take the limit in $n$ to obtain $f_{\dR*}Z \times_V \mathrm{Spf}(\widehat{\ca{O}}_{V,v})$ which is a disjoint union of projective formal $\widehat{\ca{O}}_{V,v}$-schemes. Note that this is simply notation, since there is no actual map $\mathrm{Spf}(\widehat{\ca{O}}_{V,v}) \rightarrow V$ of locally ringed spaces.

\begin{proposition} \label{prop: bc2} There is a natural isomorphism 
\[ f_{\dR*}Z \times_V \mathrm{Spf}(\widehat{\ca{O}}_{V,v}) \isomto f_{v,\dR*}\widehat{Z}_v  \]
of disjoint unions of projective formal schemes over $\widehat{\ca{O}}_{V,v}$, such that the diagram
\[ \xymatrix{  f^*f_{\dR*}Z \times_V \mathrm{Spf}(\widehat{\ca{O}}_{V,v}) \ar[r] \ar[d]_{\varepsilon_Z \times_V\mathrm{Spf}(\widehat{\ca{O}}_{V,v}) } & f_v^*f_{v,\dR*}\widehat{Z}_v \ar[d]^{\varepsilon_{\widehat{Z}_v}} \\  Z \times_V \mathrm{Spf}(\widehat{\ca{O}}_{V,v})  \ar@{=}[r]& \widehat{Z}_v   }  \]
of formal schemes with $\widehat{V}_v$-linear stratifications commutes.
\end{proposition}

\begin{proof}
This simply follows from applying Proposition \ref{prop: props}(\ref{prop: pp2}) to the various infinitesimal neighbourhoods $v^{(n)}\rightarrow V$, and then taking the limit in $n$.
\end{proof}

\begin{corollary} \label{cor: flapr} If $Z\in \mathbf{StrP}(W/K)$ is smooth over $W$, then for each connected open affinoid $\spa{A}\subset V$ the fibre product
\[ f_{\dR*}Z \times_V \spa{A} \]
is flat and projective over $\spa{A}$.
\end{corollary}

\begin{proof}
We may assume that $V=\spa{A}$, and that $f,g$ come from algebraic maps
\[ f^a:W^a\rightarrow \spec{A},\;\;g^a:Z^a\rightarrow W^a \]
of projective $A$-schemes. We have already observed in Propostion \ref{prop: rep2} that in this situation $f_{\dR*}Z $ is the analytification of a disjoint union $f^a_{\dR*}Z^a$ of projective $A$-schemes. Moreover, for any closed point $\mathfrak{m}\in \spec{A}$, we know by Theorem \ref{theo: stratfor} that the base change $f^a_{\dR*}Z^a \otimes_A \widehat{A}_\mathfrak{m}$ admits a formal stratification. In particular, by applying \cite[Lemma 6.2]{dS15}, we can see that $f^a_{\dR*}Z^a \otimes_A \widehat{A}_\mathfrak{m}$ must be flat over $\widehat{A}_\mathfrak{m}$. 

Since this is true for all $\mathfrak{m}$, it follows that $f^a_{\dR*}Z^a$ is flat over $A$, so the restriction of $f^a_{\dR*}Z^a\rightarrow \spec{A}$ to each of its connected components is flat and projective. Since $\spec{A}$ is connected, each of these components must be set theoretically surjective over $\spec{A}$, so each has a non-empty fibre over $\mathfrak{m}$. Hence the map
\[ \pi_0(f^a_{\dR*}Z^a \otimes_A A/\mathfrak{m}) \rightarrow \pi_0(f^a_{\dR*}Z^a ) \]
on connected components is surjective. Now applying \cite[Proposition 6.4]{dS15} and Proposition \ref{prop: props}(\ref{prop: pp3}) to the fibre $f^a_{\dR*}Z^a \otimes_A A/\mathfrak{m}$ we know that this it has only finitely many connected components. Therefore so does $f^a_{\dR*}Z^a$, it is thus flat and projective over $A$. The claim now follows by taking the analytification.
\end{proof}

\section{Invariance under the stratification and exactness of the homotopy sequence} \label{sec: pf1}

Having constructed the relative push-forwards in \S\ref{sec: push}, what we should do next is emulate the construction of \cite[\S13]{dS15} to endow $f_{\dR*}Z$ with a stratification, at least when $Z/W$ is smooth and comes from an object of $\mathbf{StrP}(W/K)$. We should then show that $f_{\dR*}$ is a `weak adjoint' to $f^*$. In fact, to obtain the proof of Theorem \ref{theo: main1} we can get away with the formal version of this result, namely Theorem \ref{theo: stratfor}.

Let us put ourselves in the situation of Theorem \ref{theo: main1}, so that $f:W\rightarrow V$ is smooth with geometrically connected fibres, and $V$ is smooth and geometrically connected. Fix $w\in W(K)$ and set $v=f(w)$. Let $(Y,\rho)\in \ca{R}_W$ and assume that $Y$ is smooth over $K$. Let $Z=\ca{U}_W(Y,\rho)\in \mathbf{StrP}(W/K)$. By applying the forgetful functor $\mathbf{StrP}(W/K)\rightarrow \mathbf{StrP}(W/V)$ we may construct $\varepsilon_Z: f^*f_{\dR*}Z\rightarrow Z$ as in the previous section.

\begin{proposition} The map $\varepsilon_Z:f^*f_{\dR*}Z\rightarrow Z$ is a closed immersion.
\end{proposition}

\begin{proof} We first note that by Proposition \ref{prop: props} the given map $\varepsilon_Z:f^*f_{\dR*}Z\rightarrow Z$ becomes a closed immersion on the fibre $W_v$ over the given $K$-valued point $v\in V(K)$. By letting $v$ varying and possibly increasing the base field $K$, we deduce that the same is true over any rigid point of $V$. Moreover, we know from Corollary \ref{cor: flapr} that on any connected open affinoid $\spa{A}\subset V$, the base change 
\[ f_{\dR*}Z\times_V \spa{A}\rightarrow \spa{A}\]
is projective over $A$. Since $\varepsilon_Z$ being a closed immersion is local on $V$, we therefore find ourselves in the following general situation. We have a smooth projective morphism $f:W\rightarrow \spa{A}$ over an affinoid base, and a morphism $T\rightarrow Z$ of projective $W$-varieties, which is a closed immersion after passing to any rigid point of $\spa{A}$. We wish to show that $T\rightarrow Z$ is a closed immersion. This is now a situation which can be algebrised, since by rigid analytic GAGA all three of $W,T,Z$ come from projective $A$-schemes, as do all the morphisms between them.

Since the role of $W$ can now be ignored, we can reduce to the following. Let $A$ be a Noetherian ring, and $i:T^a\rightarrow Z^a$ a morphism of projective $A$-schemes, which is a closed immersion on the fibres over all closed points of $A$. Then we wish to show that $i$ is an closed immersion. To see this, note that the quasi-finite locus of $i$ must contain every closed point of $Z^a$,  it must therefore be equal to $Z^a$ by \cite[Th\'{e}or\`{e}me 13.1.5]{EGA4.3}. Therefore $i$ is quasi-finite and projective, hence finite, say locally of the form $\spec{C}\rightarrow \spec{B}$. Moreover, for any maximal ideal $\mathfrak{m}$ of $B$, the induced map  $\frac{B}{\mathfrak{m}} \rightarrow C\otimes_B \frac{B}{\mathfrak{m}}$ is either the zero map or an isomorphism, in particular it is surjective. Hence $i$ is a closed immersion as claimed.
\end{proof}

\begin{definition} We say that a closed sub-variety $T\hookrightarrow Z$ is stable under the ($K$-linear) stratification if the composite map
\[ T \times_{W,p_0} \widehat{P}_W \rightarrow Z \times_{W,p_0} \widehat{P}_W \overset{\mathrm{pr}_Z\circ \epsilon}{\longrightarrow} Z \]
factors through $T$.
\end{definition}

\begin{theorem} \label{theo: inva}
Let $Z=\ca{U}_W(Y,\rho)$ for some $(Y,\rho) \in \ca{R}_W$ smooth over $K$. Then the closed immersion
\[ f^*f_{\dR*}Z\hookrightarrow Z \]
is stable under the $K$-linear stratification on $Z$.
\end{theorem}

\begin{proof} Let us write $T$ for $f^*f_{\dR*}Z$ to save on notation. First of all, the claim is local on $V$, which we may therefore assume to be affinoid, $V=\spa{A}$. Applying Theorem \ref{theo: stratfor} and Proposition \ref{prop: bc2} we know that for any rigid point $v\in V$ the base change $T \times_V \mathrm{Spf}(\widehat{\ca{O}}_{V,v})$ (again, this is just notation) is stable under the pull-back stratification. Now, since we are in characteristic zero, stability under the stratification (of either $T$ or $	T \times_V \mathrm{Spf}(\widehat{\ca{O}}_{V,v})$) amounts simply to stability under the induced action of $\ca{D}er_K(\ca{O}_W)$. Therefore what we need to show is that for any local section $\partial\in \ca{D}er_K(\ca{O}_W)$, and any local section $f\in \ca{I}_T$ of the ideal of $T$ in $Z$, the section $\partial(f)$ is also in $\ca{I}_T$, i.e. maps to zero in $\ca{O}_T$. 

Since $Z$ is projective over $\mathrm{Spa}(A)$, we may algebrise $Z$ to produce a projective scheme $Z^a$ over $\spec{A}$, and replace $f$ by an local section of $\mathcal{O}_{Z^a}$, as well as $\partial$ by a suitable algebraic derivation. Since $T\times_V \mathrm{Spf}(\widehat{\mathcal{O}}_{V,v})$ is stable under the connection, we know that $\partial(f)$ restricts to zero as a local section of $\mathcal{O}_{T\times_V \mathrm{Spf}(\widehat{\mathcal{O}}_{V,v})}$. Hence we are reduced to the following situation: $A$ is a regular Noetherian ring, $T$ is a flat and projective scheme over $A$,\footnote{flatness comes from Corollary \ref{cor: flapr}} and $g\in \mathcal{O}_T$ is some local section of the structure sheaf of $T$. \footnote{in our case, $g$ is the image of $\partial(f)$ under the quotient map $\mathcal{O}_Z\rightarrow \mathcal{O}_T$} We are given that, for any maximal ideal $\fr{m}$ of $A$, $g$ vanishes as a local section of $\mathcal{O}_{\widehat{T}_{/\fr{m}}}$, the structure sheaf of the formal completion of $Z$ along $\fr{m}$. We wish to show that in fact $g=0$.

If $g$ is non-zero, there exists some closed point $t\in T$ such that $g$ is non-zero in $\mathcal{O}_{T,t}$. Since $\mathcal{O}_{T,t}$ is Noetherian, the map $\mathcal{O}_{T,t}\rightarrow \widehat{\mathcal{O}}_{T,t}$ is faithfully flat, and hence $g$ is non-zero in $\widehat{\mathcal{O}}_{T,t}$. But now if $\fr{m}$ is the image of $t$ in $\spec{A}$, we have the factorisation $\mathcal{O}_{T,t}\rightarrow \mathcal{O}_{\widehat{T}_{/\fr{m}},t}\rightarrow \widehat{\mathcal{O}}_{T,t}$ and so $g$ must be non-zero in $\mathcal{O}_{\widehat{T}_{/\fr{m}},t}$, which is a contradiction. Hence we do indeed have $g=0$ as required.
\end{proof}

We can now complete the proof of Theorem \ref{theo: main1}.

\begin{proof}[Proof of Theorem \ref{theo: main1}]
Suppose that we have some $E\in \mathrm{MIC}(W/K)$. Then applying Theorem \ref{theo: inva} we obtain a closed sub-variety
\[ f^*f_{\dR*}\P(E) \hookrightarrow \P(E) \]
which is stable under the stratification, and by Proposition \ref{prop: props} recovers the inclusion
\[ \P(E_w)^{\pi_1^\dR(W_v,w)} \hookrightarrow \P(E_w). \]
on the fibre over $w$. Since $f^*f_{\dR*}\P(E)$ is stable under the stratification on $\P(E)$, it therefore acquires an induced stratification such that 
\[  f^*f_{\dR*}\P(E) \hookrightarrow \P(E)  \]
is a closed immersion of stratified $W$-varieties. We now apply Theorem \ref{theo: tanninv} to deduce that the closed immersion $f^*f_{\dR*}\P(E) \hookrightarrow \P(E)$ must come from a unique $\pi_1^\dR(W,w)$-invariant closed sub-scheme of $\P(E_w)$. Put differently, we can see that the closed sub-scheme $\P(E_w)^{\pi_1^\dR(W_v,w)} \subset \P(E_w)$ is invariant under $\pi_1^\dR(W,w)$. Hence we may conclude by applying Proposition \ref{prop: mix}.
\end{proof}

\section{The homotopy sequence for algebraic \texorpdfstring{$k$}{k}-varieties}

In the second part of this article, we will use Theorem \ref{theo: main1} to deduce a corresponding result for algebraic $k$-varieties. The particular positive characteristic analogue of the category $\mathrm{MIC}(W/K)$ that we will be interested in is the category $\mathrm{Isoc}^\dagger(X/K)$ of overconvergent isocrystals on a $k$-variety $X$, and to begin with we will quickly recall the basic definitions and constructions. The interested reader should consult \cite{Ber96b,LS07} for more details.

\begin{definition} \label{defn: frame} A \emph{frame} over $\ca{V}$ is a triple $(X,Y,\mathfrak{P})$ consisting of an open immersion $X\hookrightarrow Y$ of $k$-varieties and a closed immersion $Y\hookrightarrow \mathfrak{P}$ of formal $\ca{V}$-schemes of finite type. We say that it is proper if $Y$ is proper over $k$ and smooth if $\mathfrak{P}$ is smooth over $\ca{V}$ in a neighbourhood of $X$.
\end{definition}

In this situation we have a (continuous) specialisation map
\[ \mathrm{sp}:\mathfrak{P}_K \rightarrow \mathfrak{P} \]
where $\mathfrak{P}_K$ is the generic fibre of $\mathfrak{P}$, considered as an (adic) analytic variety over $K$. We define the tube \[ ]Y[_\mathfrak{P}:= \mathrm{sp}^{-1}(Y)^\circ \]
of $Y$ to be the topological interior of the inverse image under $\mathrm{sp}$. As an open subset of $\mathfrak{P}_K$ the tube $]Y[_\mathfrak{P}$ naturally admits the structure of an analytic variety over $K$. There is moreover an induced map $\mathrm{sp}_Y:]Y[_\mathfrak{P}\rightarrow Y$ and we define the tube
\[ ]X[_\mathfrak{P}:= \overline{\mathrm{sp}^{-1}_Y(X)} \]
of $X$ to be the topological closure of the inverse image under $\mathrm{sp}_Y$. Note that in general $]X[_\mathfrak{P}$ is not necessarily an adic space, simply a closed subset of one. We let $j:]X[_\mathfrak{P}\rightarrow ]Y[_\mathfrak{P}$ denote the natural inclusion, and for any sheaf $\ca{F}$ on $]Y[_\mathfrak{P}$ we define
\[ j_X^\dagger\ca{F}:= j_*j^{-1}\ca{F} .\]
We let $ \mathrm{MIC}(j_X^\dagger\ca{O}_{]Y[_\mathfrak{P}})$ denote the category of coherent $j_X^\dagger\ca{O}_{]Y[_\mathfrak{P}}$-modules with integrable connection. If $(X,Y,\mathfrak{P})$ is a smooth and proper frame, then by definition
\[ \mathrm{Isoc}^\dagger(X/K) \subset \mathrm{MIC}(j_X^\dagger\ca{O}_{]Y[_\mathfrak{P}}) \]
is the full subcategory consisting of those objects which satisfy a suitable convergence condition on their Taylor isomorphisms. We will not describe this condition explicitly here, but it can be found for example in \cite[Theorem 4.3.9]{LS07}. We can summarise the main properties of $\mathrm{Isoc}^\dagger(X/K)$ as follows, again proofs of all these claims can be found in either of \cite{Ber96b,LS07}.

\begin{theoremu} The category $\mathrm{Isoc}^\dagger(X/K)$ does not depend on the choice of smooth and proper frame $(X,Y,\mathfrak{P})$, is functorial in $X$ and $K$, and glues over a Zariski open cover of $X$ (it is therefore defined for varieties which do not necessarily admit such a frame). 
\end{theoremu}

Then $\mathrm{Isoc}^\dagger(X/K)$ is the category of local systems for the theory of rigid cohomology. When $X$ is geometrically connected, it is Tannakian over $K$ by \cite[Lemma 1.8]{Cre92}. Moreover, if $x\in X(k)$ is a rational point, then $\mathrm{Isoc}^\dagger(X/K)$ is neutral Tannakian over $K$ and
\[ x^*:\mathrm{Isoc}^\dagger(X/K) \rightarrow \mathrm{Vec}_K \]
provides a fibre functor.

\begin{definition} We define the overconvergent fundamental group $\pi_1^\dagger(X,x)$ of $X$ with base-point $x$ to be the Tannaka dual of $\mathrm{Isoc}^\dagger(X/K)$ with respect to $x^*$. It is an affine group scheme over $K$.
\end{definition}

For brevity, we have not included $K$ in the notation for the overconvergent fundamental group, this should not present to much of a problem. Suppose now that we are given a morphism $f:X\rightarrow S$ of geometrically connected $k$-varieties, and $x\in X(k)$. Write $s=f(x)$. If the fibre $X_s$ is also geometrically connected, then we call the sequence
\[\pi_1^\dagger(X_s,x) \rightarrow\pi_1^\dagger(X,x)\rightarrow\pi_1^\dagger(S,s)\rightarrow 1\]
the \emph{homotopy sequence} associated to the pair $(f,x)$. The positive characteristic analogue of Theorem \ref{theo: main1} is then the following.

\begin{theorem} \label{theo: main2}
Let $f:X\rightarrow S$ be a smooth and projective morphism of smooth $k$-varieties, with geometrically connected fibres and base, and let $x\in X(k)$. Then the homotopy sequence of the pair $(f,x)$ is exact.
\end{theorem}

The basic idea of the proof will be to reduce to the situation when $X$ is a family of smooth projective curves. In this case we may lift the whole family to characteristic $0$, and then standard results in rigid cohomology will enable us to deduce the exactness we require from Theorem \ref{theo: main1}. Let us note now that formation of $\pi_1^\dagger$ commutes with taking finite extension of $K$ (and hence of $k$), and exactness of a sequence of affine group schemes can be checked after such a finite extension. We will make use of this to freely take finite extensions of $k$ throughout the proof.

Our first task in the proof of Theorem \ref{theo: main2} will be to show the overconvergent analogue of Lemma \ref{lemma: dradj}, i.e. that the homotopy sequence of a smooth projective morphism is always weakly exact. While conceptually simple, the proof will require the existence of push-forwards for overconvergent isocrystals, and showing this will require some rather daunting heavy machinery from the theory of arithmetic $\ca{D}^\dagger$-modules.

\begin{proposition} \label{prop: adjoint} Let $f:X\rightarrow S$ be a smooth and projective morphism of smooth $k$-varieties, of constant relative dimension $d$. Then there exists a functor
\[ f_*:\mathrm{Isoc}^\dagger(X/K) \rightarrow \mathrm{Isoc}^\dagger(S/K), \]
right adjoint to 
\[f^*:\mathrm{Isoc}^\dagger(S/K) \rightarrow \mathrm{Isoc}^\dagger(X/K), \]
such that for any $s\in S(k)$ we have $s^*f_*E \cong H^0_\rig(X_s/K,E|_{X_s})$.
\end{proposition}

The proof will be postponed until \S\ref{sec: adjoint} below, and for now we will deduce several important consequences.

\begin{corollary} \label{cor: weakhes} Let $f:X\rightarrow S$, $x\in X(k)$, $s=f(x)$ be as in Theorem \ref{theo: main2}. Then the sequence
\[ \pi_1^\dagger(X_s,x)\rightarrow \pi_1^\dagger(X,x)\rightarrow \pi_1^\dagger(S,s)\rightarrow 1\]
of affine group schemes is weakly exact.
\end{corollary}

\begin{proof} We will apply Theorem \ref{theo: tann}. Given Proposition \ref{prop: adjoint}, the proof is identical to Lemma \ref{lemma: dradj} above.
 \end{proof}


This will enable us to make some rather major simplifying assumptions in the proof of Theorem \ref{theo: main2}, in particular by combining it with the following elementary lemma.

\begin{lemma} \label{lemma: 6lem} Let
\[ \xymatrix{ K_0 \ar[r]\ar@{->>}[d] & G_0 \ar[r]\ar@{->>}[d] & H_0 \ar[r]\ar@{->>}[d] & 1\\ K_1\ar[r] & G_1 \ar[r] & H_1 \ar[r] & 1 } \]
be a commutative diagram of sequences of affine group schemes, with surjective vertical maps. Assume that the top sequence is exact, and that the bottom sequence is weakly exact. Then the bottom sequence is in fact exact.
\end{lemma}

\begin{proof} Follows from the fact that the image of a normal subgroup under a surjective group homomorphism is again a normal subgroup. 
\end{proof}

Another corollary of weak exactness that will play an important role is the following very weak version of the Lefschetz hyperplane theorem for overconvergent fundamental groups.

\begin{theorem} \label{theo: weaklef} Let $X\subset \P^n_k$ be smooth, projective and geometrically connected. Assume that $\dim X\geq 2$ and let $Y=H\cap X$ be a smooth hyperplane section. Let $y\in Y(k)$. Then the normal closure of the image of 
\[ \pi_1^\dagger(Y,y)\rightarrow \pi_1^\dagger(X,y)\]
is the whole of $\pi_1^\dagger(X,y)$.
\end{theorem}

\begin{proof} After possibly extending $k$ we may assume (by Bertini's hyperplane section theorem) that there exists a hyperplane $H'\subset\mathbb{P}^n_k$ distinct from $H$ such that $y\in Y\cap H'$ and $Z=Y\cap H'$ is smooth. Let $g:\widetilde{X}\rightarrow X$ be the blow-up of $X$ along $Z$, and $\widetilde{Y}\subset \widetilde{X}$ the strict transform of $Y$. Thus $\widetilde{Y}\cong Y$ and in particular we may lift $y$ canonically to a point $\tilde{y}$ of $\widetilde{Y}$. 

\begin{claimu} The induced map $g_*:\pi_1^{\dagger}(\widetilde{X},\tilde{y})\rightarrow\pi_1^{\dagger}(X,y) $ is an isomorphism.
\end{claimu}

\begin{proof}[Proof of claim] We want to show that the functor $\mathrm{Isoc}^\dagger(X/K)\rightarrow \mathrm{Isoc}^\dagger(\widetilde{X}/K)$ is an equivalence of categories. According to \cite[Proposition 5.3.6]{Ked07} the functor is essentially surjective. It is automatically faithful, hence we must demonstrate that it is full. So let $g^*E\rightarrow g^*F$ be a morphism. Since $g:\widetilde{X}\setminus g^{-1}(Z)\rightarrow X\setminus Z$ is an isomorphism, this induces a morphism $E|_{X\setminus Z}\rightarrow F|_{X\setminus Z}$ which by \cite[Theorem 5.2.1]{Ked07} must come from a morphism $E\rightarrow F$.
 \end{proof}

Clearly, the claim also holds for $\widetilde{Y}\rightarrow Y$, hence it suffices to show that the normal closure of the image of
\[ \pi_1^{\dagger}(\widetilde{Y},\tilde{y})\rightarrow \pi_1^{\dagger}(\widetilde{X},\tilde{y}) \]
is the whole of $\pi_1^{\dagger}(\widetilde{X},\tilde{y})$. The pencil of hyperplane sections spanned by $Y=X\cap H$ and $X\cap H'$ furnishes a projective map $a:\widetilde{X}\rightarrow \mathbb{P}^1_k$ whose generic fibre is smooth, and the pre-image of $y$ with respect to $g$ is a closed sub-scheme of $\widetilde{X}$ mapped isomorphically onto $\P^1_k$ by $a$. Thus we obtain a section 
$\sigma$ of $a$, and it is easy to check that $\sigma(a(\tilde{y}))=\tilde{y}$. 

Let $b:\widetilde{X}_U\rightarrow U$ be the smooth locus of $a$, note that $\widetilde{Y}\subset \widetilde{X}_U$ is a fibre of $b$. Since $b$ has geometrically connected fibres, Corollary \ref{cor: weakhes} implies that the induced sequence of group schemes
\[ \pi_1^{\dagger}(\widetilde{Y},\tilde{y})\rightarrow \pi_1^\dagger(\widetilde{X}_U,\tilde y) \rightarrow \pi_1^{\dagger}(U,a(\tilde{y}))\rightarrow 1 \]
is weakly exact. We also have a commutative diagram of affine group schemes
\[ \xymatrix{  \pi_1^{\dagger}(\widetilde{Y},\tilde{y})\ar[r]\ar@{=}[d] & \pi_1^\dagger(\widetilde{X}_U,\tilde y)\ar[r]\ar@{->>}[d]& \pi_1^{\dagger}(U,a(\tilde{y}))\ar[r]\ar[d]\ar@/_1.5pc/[l]_{\sigma_*} & 1  \\  \pi_1^{\dagger}(\widetilde{Y},\tilde{y})\ar[r] & \pi_1^\dagger(\widetilde X,\tilde y)\ar[r]& \pi_1^{\dagger}(\mathbb{P}^1_k,a(\tilde{y}))\ar[r]\ar@/_1.5pc/[l]_{\sigma_*} & 1 }
\]
such that the middle vertical arrow arrow is surjective by \cite[Theorem 5.2.1, Proposition 5.3.1]{Ked07}. 

Note that $\pi_1^\dagger(\P^1_k,a(\tilde y))=\{1\}$ by Lemma \ref{lemma: projbun} below. Hence applying Theorem \ref{theo: tann} with $A=\{1\}$  we must show that:
\begin{enumerate} \item an object $V\in \mathrm{Rep}_K(\pi_1^\dagger(\widetilde X,\tilde y))$ is trivial if and only if it is trivial in $\mathrm{Rep}_K(\pi_1^\dagger(\widetilde Y,\tilde y))$;
\item for any object $V\in \mathrm{Rep}_K(\pi_1^\dagger(\widetilde X,\tilde y))$, $V^{\pi_1^\dagger(\widetilde Y,\tilde y)}$ is stable under $\pi_1^\dagger(\widetilde X,\tilde y)$.
\end{enumerate}
Note that the `only if' part of (1) is trivial, and in the situation of (2) it follows from weak exactness of 
\[ \pi_1^{\dagger}(\widetilde{Y},\tilde{y})\rightarrow \pi_1^\dagger(\widetilde{X}_U,\tilde y) \rightarrow \pi_1^{\dagger}(U,a(\tilde{y}))\rightarrow 1 \]
that $V^{\pi_1^\dagger(\widetilde Y,\tilde y)}$ is stable under $\pi_1^\dagger(\widetilde{X}_U,\tilde y)$. By surjectivity of $\pi_1^\dagger(\widetilde{X}_U,\tilde y)\rightarrow \pi_1^\dagger(\widetilde{X},\tilde y)$ it is therefore stable under $\pi_1^\dagger(\widetilde{X},\tilde y)$. 

It remains to show the `if' part of (1), therefore let us assume that we have some $V\in \mathrm{Rep}_K(\pi_1^\dagger(\widetilde X,\tilde y))$ on which $\pi_1^\dagger(\widetilde Y,\tilde y)$ acts trivially. Let $H$ be the normal closure of the image of
\[ \pi_1^\dagger(\widetilde Y,\tilde y) \rightarrow \pi_1^\dagger(\widetilde{X}_U,\tilde y),  \]
we therefore know that:
\begin{itemize} \item $\pi_1^\dagger(\widetilde{X}_U,\tilde y)= H\cdot \sigma_*(\pi_1^\dagger(U,a(\tilde y)))$;
\item $H$ acts trivially on $V$. 
\end{itemize} 
Since $\sigma_*(\pi_1^\dagger(U,a(\tilde y)))$ has trivial image in $\pi_1^\dagger(\widetilde{X},\tilde{y})$ it follows that $H\rightarrow \pi_1^\dagger(\widetilde{X},\tilde{y})$ is surjective, hence $\pi_1^\dagger(\widetilde X,\tilde y)$ acts trivially on $V$ as required.
 \end{proof}

\begin{lemma} \label{lemma: projbun} Let $f:Y\rightarrow Z$ be a map which Zariski-locally on $Z$ is a product of projective bundles. Then the induced map
\[ f^*:\mathrm{Isoc}^\dagger(Z/K) \rightarrow \mathrm{Isoc}^\dagger(Y/K) \]
is an equivalence of categories.
\end{lemma}

\begin{proof} Applying Proposition \ref{prop: adjoint} we get unit and counit maps
\begin{align*} f^*f_*E &\rightarrow E,\;\;E\in \mathrm{Isoc}^\dagger(Y/K) \\
F &\rightarrow f_*f^*F,\;\;F\in \mathrm{Isoc}^\dagger(Z/K),
\end{align*}
we must prove that these are isomorphisms. But this can be checked fibre by fibre, we are therefore reduced to the case of the structure map $\P^{n_1}_k \times_k \ldots \times_k \P^{n_r}_k \rightarrow \spec{k}$. Hence using GAGA we may reduce to the statement that over a field $K$ of characteristic $0$, every vector bundle with integrable connection on $\P^{n_1}_K \times_K \ldots \times_K \P^{n_r}_K$ is trivial.
 \end{proof}

\section{Pushforward of overconvergent isocrystals}\label{sec: adjoint}

The purpose of this section is to prove Proposition \ref{prop: adjoint}, stating that if $f:X\rightarrow S$ is a smooth and proper morphism of $k$-varieties, then $f^*:\mathrm{Isoc}^\dagger(S/K)\rightarrow \mathrm{Isoc}^\dagger(X/K)$ has a right adjoint $f_*:\mathrm{Isoc}^\dagger(X/K)\rightarrow \mathrm{Isoc}^\dagger(S/K)$, which commutes with base change and on fibres recovers the zeroeth cohomology group $H^0_\rig$ with coefficients. This will require some heavy machinery from the theory of arithmetic $\ca{D}$-modules, as developed by Berthelot and Caro, and while the result we require is essentially contained in work of Caro, it will take a little care to extract it in the form that we need. This section is rather technical, and the casual reader will gain little from going through it in detail; they are advised to simply take Proposition \ref{prop: adjoint} on trust.

With these warnings out of the way, let us begin. To start with, by uniqueness of adjoints, the question is local on the base $S$, which we may therefore assume to be affine. Hence we may assume that there exists a smooth and proper formal scheme $\mathfrak{Q}$ over $W$, a divisor $H\subset \mathfrak{Q} \times_\ca{V} k$ and a locally closed immersion $S\hookrightarrow \mathfrak{Q}$ such that $S=\overline{S}\setminus H$. In other words, $(\mathfrak{Q},H,\overline{S})$ is a `triplet lisse en dehors du diviseur' in the sense of \cite[D\'efinition 3.1.6]{Car11}. By choosing a projective embedding of $X$ over $S$, we may construct another smooth and proper formal scheme $\mathfrak{P}$ over $W$ together with a commutative diagram of embeddings
\[ \xymatrix{ X\ar[r]\ar[d]_f & \mathfrak{P} \ar[d]^g \\ S\ar[r] & \mathfrak{Q}  } \]
such that $\mathfrak{P}\rightarrow \mathfrak{Q}$ is smooth, and $X=\overline{X}\setminus g^{-1}(H)$. Set $T=g^{-1}(H)$, so that again $(\mathfrak{P},T,\overline{X})$ is a `triplet lisse en dehors du diviseur'.

We will let $D^b_\mathrm{surcoh}(\mathfrak{P},T,\overline{X}/K)$ denote the category of bounded complexes of overcoherent $\ca{D}^\dagger$-modules on the triple $(\mathfrak{P},T,\overline{X})$ in the sense of \cite[Notations 1.2.3]{Car15a}, in other words overcoherent complexes of $\ca{D}^\dagger_{\mathfrak{P},\Q}(^\dagger T)$-modules $\ca{E}$ such that $\mathbf{R}\underline{\Gamma}^\dagger_{\overline{X}}(\ca{E})\cong \ca{E}$. We will similarly denote by $D^b_\mathrm{surcoh}(\mathfrak{Q},H,\overline{S}/K)$ the category of bounded complexes of overcoherent $\ca{D}^\dagger$-modules on the triple $(\mathfrak{Q},H,\overline{S})$. Following \cite[D\'{e}finition 1.2.5]{Car15a}, we have full subcategories
\begin{align*}
 \mathrm{Isoc}^{\dagger\dagger}(\mathfrak{P},T,\overline{X}/K) &\subset \mathrm{Surcoh}(\mathfrak{P},T,\overline{X}/K) \\
 \mathrm{Isoc}^{\dagger\dagger}(\mathfrak{Q},H,\overline{S}/K) &\subset \mathrm{Surcoh}(\mathfrak{Q},H,\overline{S}/K)
\end{align*}
consisting of `overcoherent isocrystals' and by \cite[Corollaire 3.5.10, Th\'eor\`eme 4.2.2]{Car11} canonical equivalences of categories
\begin{align*}
\mathrm{sp}_{\overline{X}\hookrightarrow \mathfrak{P},T,+}:\mathrm{Isoc}^\dagger(X/K) &\isomto \mathrm{Isoc}^{\dagger\dagger}(\mathfrak{P},T,\overline{X}/K) \\
\mathrm{sp}_{\overline{S}\hookrightarrow \mathfrak{Q},H,+}:\mathrm{Isoc}^\dagger(S/K) &\isomto \mathrm{Isoc}^{\dagger\dagger}(\mathfrak{Q},H,\overline{S}/K),
\end{align*}
we will denote inverse functors by $\mathrm{sp}^{-1}_{-}$. There are full subcategories
\begin{align*}
D^b_\mathrm{isoc}(\mathfrak{P},T,\overline{X}/K) &\subset D^b_\mathrm{surcoh}(\mathfrak{P},T,\overline{X}/K) \\
D^b_\mathrm{isoc}(\mathfrak{Q},H,\overline{S}/K) &\subset D^b_\mathrm{surcoh}(\mathfrak{Q},H,\overline{S}/K)
\end{align*}
consisting of objects whose cohomology sheaves are overcoherent isocrystals.

Let $d$ denote the relative dimension of $f$, $\mathbf{D}_T$ (resp. $\mathbf{D}_H$) the $\ca{D}^\dagger_{\mathfrak{P},\Q}(^\dagger T)$-linear (resp. $\ca{D}^\dagger_{\mathfrak{Q},\Q}(^\dagger H)$-linear) dual functor, and $g_+, g^!$ the direct and inverse image functors between $\ca{D}^\dagger_{\mathfrak{P},\Q}(^\dagger T)$ and $\ca{D}^\dagger_{\mathfrak{Q},\Q}(^\dagger H)$-modules. Applying \cite[Proposition 3.1.7, Corollaire 3.5.10]{Car11} and \cite[Th\'eor\`eme 3.3.1]{Car15a} we have factorisations
\begin{align*}
f^!:=\mathbf{R}\underline{\Gamma}^\dagger_{\overline{X}}\circ g^! : D^b_\mathrm{isoc}(\mathfrak{Q},H,\overline{S}/K)&\rightarrow  D^b_\mathrm{isoc}(\mathfrak{P},T,\overline{X}/K) \\
f_+:=g_+: D^b_\mathrm{isoc}(\mathfrak{P},T,\overline{X}/K)&\rightarrow  D^b_\mathrm{isoc}(\mathfrak{Q},H,\overline{S}/K) \\
\mathbf{D}:=\mathbf{D}_T:D^b_\mathrm{isoc}(\mathfrak{P},T,\overline{X}/K)&\rightarrow  D^b_\mathrm{isoc}(\mathfrak{P},T,\overline{X}/K) \\
\mathbf{D}:=\mathbf{D}_H:D^b_\mathrm{isoc}(\mathfrak{Q},H,\overline{S}/K)&\rightarrow  D^b_\mathrm{isoc}(\mathfrak{Q},H,\overline{S}/K). 
\end{align*}
Set $f^+:=\mathbf{D}\circ f^!\circ \mathbf{D}$. By \cite[Proposition 4.2.4]{Car11} the diagram
\[\xymatrix{\mathrm{Isoc}^\dagger(S/K) \ar[r]^-{\mathrm{sp}_{\overline{S}\hookrightarrow \mathfrak{Q},H,+}}\ar[d]_{f^*} & D^b_\mathrm{isoc}(\mathfrak{Q},H,\overline{S}/K) \ar[d]^{f^+[d]} \\
\mathrm{Isoc}^\dagger(X/K) \ar[r]^-{\mathrm{sp}_{\overline{X}\hookrightarrow \mathfrak{P},D,+}} &D^b_\mathrm{isoc}(\mathfrak{P},D,\overline{X}/K)  } \]
is 2-commutative, and by combining \cite[Th\'{e}or\`{e}mes 1.2.7, 1.2.9]{Car06b} we can see that $f_+$ and $f^+$ are adjoint functors. Putting this all together, we obtain a natural isomorphism
\[ \mathrm{Hom}_{\mathrm{Isoc}^\dagger(X/K)}(f^*F,E)\cong \mathrm{Hom}_{D^b_\mathrm{isoc}(\mathfrak{Q},H,\overline{S}/K)}(\mathrm{sp}_{\overline{S}\hookrightarrow \mathfrak{Q},H,+}F,f_+\mathrm{sp}_{\overline{X}\hookrightarrow \mathfrak{P},T,+}E[-d]) \]
for any $E\in \mathrm{Isoc}^\dagger(X/K)$ and $F\in \mathrm{Isoc}^\dagger(S/K)$. To complete the proof of Proposition \ref{prop: adjoint}, it suffices to show that:
\begin{enumerate} \item $f_+\mathrm{sp}_{\overline{X}\hookrightarrow \mathfrak{P},D,+}E[-d]$ is concentrated in degrees $\geq0$;
\item we have
\[ s^*\mathrm{sp}^{-1}_{\overline{S}\hookrightarrow \mathfrak{Q},H}\ca{H}^{-d}(f_+\mathrm{sp}_{\overline{X}\hookrightarrow \mathfrak{P},D,+}E) \cong H^0_\rig(X_s/K,E|_{X_s}) \]
\end{enumerate}
Indeed, this will imply that
\begin{align*} \mathrm{Hom}_{D^b_\mathrm{isoc}(\mathfrak{Q},H,\overline{S})}(\mathrm{sp}_{\overline{S}\hookrightarrow \mathfrak{Q},H,+}F,&f_+\mathrm{sp}_{\overline{X}\hookrightarrow \mathfrak{P},D,+}E[-d])  \\ &= \mathrm{Hom}_{\mathrm{Isoc}^\dagger(S/K)}(F,\mathrm{sp}^{-1}_{\overline{S}\hookrightarrow \mathfrak{Q},H}\ca{H}^{-d}(f_+\mathrm{sp}_{\overline{X}\hookrightarrow \mathfrak{Q},D,+}E)),\end{align*}
and therefore taking 
\[ f_*E=\mathrm{sp}^{-1}_{\overline{S}\hookrightarrow \mathfrak{Q},H}\ca{H}^{-d}(f_+\mathrm{sp}_{\overline{X}\hookrightarrow \mathfrak{P},D,+}E) \]
will do the trick. To prove these, we note that by combining \cite[Proposition 4.2.4, Th\'eor\`eme 5.2.5]{Car11} and \cite[Th\'eor\`eme 4.4.2]{Car15a}, we have an isomorphism
\[ s^*\mathrm{sp}^{-1}_{\overline{S}\hookrightarrow \mathfrak{Q},H}\ca{H}^{i-d}(f_+\mathrm{sp}_{\overline{X}\hookrightarrow \mathfrak{P},D,+}E) \cong H^{i-d}(f_{s,+}\mathrm{sp}_{X_s\hookrightarrow \mathfrak{P}_{\tilde{s}},+}E|_{X_s}),  \]
where $\tilde{s}$ denotes some lift of $s$ to a $\ca{V}$-point of $\mathfrak{Q}$, and $\mathfrak{P}_{\tilde s}$ the fibre of $g$ over $\tilde{s}$. We may therefore reduce to Lemma \ref{lemma: conv} below. 

\begin{lemma}\label{lemma: conv} Let $X\hookrightarrow \mathfrak{P}$ be a closed embedding of a smooth $d$-dimensional $k$-variety into a smooth formal $\ca{V}$-scheme, and let $f:\mathfrak{P}\rightarrow \spf{\ca{V}}$ denote the structure morphism. Let $E$ be a convergent isocrystal on $X$. Then for all $i\in \Z$ we have
\[H^{i-d}(f_+ \mathrm{sp}_{X\hookrightarrow \mathfrak{P},+}E) \cong \begin{cases} H^0_\mathrm{conv}(X/K,E)  & i=0 \\
 0 & i<0. \end{cases} \]
\end{lemma}

\begin{proof} By considering suitable spectral sequences arising from an open cover of $\mathfrak{P}$, the question is local on $\mathfrak{P}$, so we may reduce to the corresponding question for both $\mathfrak{P}$ and $X$ affine, and we may therefore assume that $X\hookrightarrow \mathfrak{P}$ lifts to a closed embedding of smooth formal $\ca{V}$-schemes $i:\mathfrak{X}\hookrightarrow \mathfrak{P}$. 

In this case, we can describe Caro's functor $\mathrm{sp}_{X\hookrightarrow \mathfrak{P},+}$, constructed in \cite[\S2.5]{Car09a}, very explicitly. Indeed, what Caro does is take an open cover of $\mathfrak{P}$ by affines $\mathfrak{P}_\alpha$, and chooses a lift of each inclusion $X\cap \mathfrak{P}_\alpha \hookrightarrow \mathfrak{P}_\alpha$ to a closed immersion of smooth formal schemes $i_\alpha: \mathfrak{X}_\alpha\rightarrow \mathfrak{P}_\alpha$. The restriction of $E$ to each $X\cap \mathfrak{P}_\alpha$ is then realised on $\fr{X}_{\alpha K}$ as a module with convergent integrable connection $E_{\mathfrak{X}_\alpha}$. Thus $\mathrm{sp}_*E_{\mathfrak{X}_\alpha}$ is a coherent $\mathscr{D}^\dagger_{\mathfrak{X}_\alpha\Q}$-module by \cite[Proposition 4.1.4]{Ber96a}, and Caro shows that the modules $i_{\alpha+}\mathrm{sp}_*E_{\mathfrak{X}_\alpha}$ glue together to give a coherent $\mathscr{D}^\dagger_{\mathfrak{P}\Q}$-module $\mathrm{sp}_{X\hookrightarrow \mathfrak{P},+}E$. (Note that in our situation, the `divisor at $\infty$', denoted by $T$ in \cite{Car09a}, is empty.)

Since we are already in the affine case, the whole construction simplifies to give $\mathrm{sp}_{X\hookrightarrow \mathfrak{P},+}E\cong i_+\mathrm{sp}_*E_\mathfrak{X}$, where $E_{\fr{X}}$ denotes the realisation of $E$ on $\mathfrak{X}_K$. If we therefore let $g:\mathfrak{X}\rightarrow \spf{\mathcal{V}}$ denote the structure morphism of $\fr{X}$, we obtain $f_+\mathrm{sp}_{X\hookrightarrow \mathfrak{P},+}E\cong g_+\mathrm{sp}_*E_\mathfrak{X}$ by \cite[(4.3.6.1)]{Ber02}. Finally, the overconvergent Spencer resolution \cite[(4.2.1.1)]{Ber02} shows that
\[ g_+\mathrm{sp}_*E_\mathfrak{X}\cong \mathbf{R}\Gamma(\fr{X},\mathrm{sp}_*E_\mathfrak{X}\otimes_{\mathcal{O}_{\fr{X}}} \Omega^\bullet_{\fr{X}})[d] \cong \mathbf{R}\Gamma(\fr{X}_K,E_\mathfrak{X}\otimes_{\mathcal{O}_{\fr{X}_K}} \Omega^\bullet_{\fr{X}_K})[d]= \mathbf{R}\Gamma_\mathrm{conv}(X/K,E)[d] \]
as required.
\end{proof}

\section{Exactness for liftable morphisms}\label{sec: exactlift}

In this and the following sections, we will build up to the proof of Theorem \ref{theo: main2} in stages, starting from very particular situations and then reducing the general case to these. The first situation in which we will prove Theorem \ref{theo: main2} is under some very strong liftability assumptions on the morphism $f$.

So suppose that we have some smooth affine variety $S=\spec{A_0}$ over $k$. Then by \cite[Th\'{e}or\`{e}me 6]{Elk73} we know that $S$ lifts to a smooth and affine $\ca{V}$-scheme $\ca{S}=\spec{A}$. Choosing a presentation of $A$ gives us embeddings
\[ \ca{S} \hookrightarrow \A^n_\ca{V} \hookrightarrow \P^n_\ca{V} \]
and we let $\mathfrak{S}$ denote the completion of the closure of $\ca{S}$ inside $\P^N_\ca{V}$. Let $\overline{S}$ denote the closure of $S$ inside $\mathfrak{S}\times_\ca{V} k$, so we have a smooth and proper frame $(S,\overline{S},\mathfrak{S})$ over $\ca{V}$, in the sense of Definition \ref{defn: frame}.

\begin{definition} We will call any frame of the form $(S,\overline{S},\mathfrak{S})$, as just constructed, a Monsky--Washnitzer frame.
\end{definition}

The main result of this section is then the following.

\begin{theorem}\label{theo: heslift} Let $f:X\rightarrow S$, $x\in X(k)$ and $s=f(x)$ be as in Theorem \ref{theo: main2}. Assume that $S$ is affine, and that there exists a morphism of frames 
\[\xymatrix{ X\ar[r] \ar[d]_f & \overline{X} \ar[r] \ar[d] & \mathfrak{X} \ar[d]^g \\ S\ar[r] & \overline{S} \ar[r] & \mathfrak{S} }  \]
extending $f$ such that:
\begin{enumerate}
\item $(S,\overline{S},\mathfrak{S})$ is a Monsky--Washnitzer frame;
\item both squares in the above diagram are Cartesian;
\item $\mathfrak{X}\rightarrow \mathfrak{S}$ is projective, and smooth in a neighbourhood of $X$.
\end{enumerate}
Then the homotopy sequence for $(f,x)$ is exact.
\end{theorem}

\begin{remark} Note that by GFGA \cite[\S5.1]{EGA3.I} together with the Monsky--Washnitzer assumption on $(S,\overline{S},\mathfrak{S})$, the morphism $g:\mathfrak{X}\rightarrow\mathfrak{S}$ in the statement of the theorem arises as the formal completion of a morphism of projective $\ca{V}$-schemes.
\end{remark}

In the situation of Theorem \ref{theo: heslift} we may choose a cofinal system of neighbourhoods $V_\lambda$ of $]S[_\mathfrak{S}$ inside $\mathfrak{S}_K$, for $\lambda \rightarrow 1^-$, such that each $V_\lambda$ is  smooth and  geometrically connected over $K$. Since the diagram
\[ \xymatrix{ X\ar[r]\ar[d] & \mathfrak{X} \ar[d] \\ S\ar[r] & \mathfrak{S}  } \]
is Cartesian, it follows that $W_\lambda:=g^{-1}(V_\lambda)$ form a cofinal system of neighbourhoods of $]X[_\mathfrak{X}$ inside $\mathfrak{X}_K$. Again, by choosing $\lambda$ close enough to $1$ we can assume that all the induced maps 
\[ g : W_\lambda \rightarrow V_\lambda \]
are smooth and projective. Let
\begin{align*} \mathrm{MIC}(S,\mathfrak{S}/K)&:=\mathrm{MIC}(j_S^\dagger\ca{O}_{\mathfrak{S}_K}) \cong 2\text{-}\mathrm{colim}_\lambda\mathrm{MIC}(V_\lambda/K) \\
 \mathrm{MIC}(X,\mathfrak{X}/K)&:=\mathrm{MIC}(j_X^\dagger\ca{O}_{\mathfrak{X}_K}) \cong 2\text{-}\mathrm{colim}_\lambda\mathrm{MIC}(W_\lambda/K) 
\end{align*}
 denote the category of coherent $j_S^\dagger\ca{O}_{\mathfrak{S}_K}$-modules (resp. $j_X^\dagger\ca{O}_{\mathfrak{X}_K}$-modules) with integrable connection (for the claimed equivalences with the 2-colimit categories, see \cite[Proposition 6.1.15]{LS07}). 
   
\begin{proposition} \label{prop: tann2}Choose a lift $\tilde x\in ]X[_\mathfrak{X}(K)$ of $x\in X(k)$, and let $\tilde s=g(\tilde x)$. Then the categories $\mathrm{MIC}(X,\mathfrak{X}/K)$ and $\mathrm{MIC}(S,\mathfrak{S}/K)$ are neutral Tannakian over $K$, with fibre functors provided by $\tilde{x}^*$ and $\tilde{s}^*$ respectively.
\end{proposition}

\begin{proof} By \ref{prop: tann1} we know that the coherent $j^\dagger_X\ca{O}_{\mathfrak{X}_K}$-module underlying any object in $\mathrm{MIC}(X,\mathfrak{X}/K)$ is locally free, and similarly for $\mathrm{MIC}(S,\mathfrak{S}/K)$. The rest of the proof is word for word the same as the proof of Proposition \ref{prop: tann1}.
 \end{proof}

We will let $\pi_1^\mathrm{colim}(]X[_\mathfrak{X},\tilde{x})$ and $\pi_1^\mathrm{colim}(]S[_\mathfrak{S},\tilde{s})$ denote the corresponding Tannaka duals.

\begin{proposition} The sequence of affine group schemes 
\[ \pi_1^\mathrm{dR}(\mathfrak{X}_{K,\tilde{s}},\tilde{x})\rightarrow \pi_1^\mathrm{colim}(]X[_\mathfrak{X},\tilde{x})\rightarrow \pi_1^\mathrm{colim}(]S[_\mathfrak{S},\tilde{s}) \rightarrow 1.  \]
is exact.
\end{proposition}

\begin{proof} By combining the push-forward functors $g_{\mathrm{dR}*}:\mathrm{MIC}(W_\lambda/K)\rightarrow \mathrm{MIC}(V_\lambda/K)$ considered in the proof of Lemma \ref{lemma: dradj} it is entirely straightforward to construct a push-forward functor
\[g_{\mathrm{dR}*}:\mathrm{MIC}(X,\mathfrak{X}/K)\rightarrow\mathrm{MIC}(S,\mathfrak{S}/K) \]
which is adjoint to $g^*$, and which on fibres recovers $H^0_\mathrm{dR}$. Now arguing exactly as in the proof of Lemma \ref{lemma: dradj} we can see that the claimed sequence is weakly exact, let $K_\mathrm{colim}$ denote the kernel of
\[ \pi_1^\mathrm{colim}(]X[_\mathfrak{X},\tilde{x})\rightarrow \pi_1^\mathrm{colim}(]S[_\mathfrak{S},\tilde{s}). \]
By Theorem \ref{theo: exactds} it therefore suffices to show that for any $E\in \mathrm{MIC}(X,\mathfrak{X}/K)$, with associated monodromy representation
\[\pi_1^\mathrm{colim}(]X[_\mathfrak{X},\tilde{x})\rightarrow \mathrm{GL}(E_{\tilde{x}}), \]
the inclusion
\[  \P(E_{\tilde{x}})^{K_\mathrm{colim}}(K)\subset \P(E_{\tilde{x}})^{\pi_1^\mathrm{dR}(\mathfrak{X}_{K,\tilde{s}},\tilde{x})}(K) \]
is in fact an equality. Note that any such object is pulled back from some $W_\lambda$ via the map $]X[_\mathfrak{X}\rightarrow W_\lambda$, let $K_\lambda$ the kernel of $ \pi_1^\mathrm{dR}(W_\lambda,\tilde{x}) \rightarrow \pi_1^\mathrm{dR}(V_\lambda,\tilde{s})$. We therefore have a natural map $K_\mathrm{colim}\rightarrow K_\lambda$. Applying Theorem \ref{theo: main1} the sequence
\[ \pi_1^\mathrm{dR}(\mathfrak{X}_{K,\tilde{s}},\tilde{x})\rightarrow \pi_1^\mathrm{dR}(W_\lambda,\tilde{x}) \rightarrow \pi_1^\mathrm{dR}(V_\lambda,\tilde{s}) \rightarrow 1  \]
is exact, and hence again applying Theorem \ref{theo: exactds} we can deduce that \[ \P(E_{\tilde{x}})^{\pi_1^\mathrm{dR}(\mathfrak{X}_{K,\tilde{s}},\tilde{x})}(K) = \P(E_{\tilde{x}})^{K_\lambda}(K).\]
But since we have $K_\mathrm{colim}\rightarrow K_\lambda$ it follows that
\[ \P(E_{\tilde{x}})^{\pi_1^\mathrm{dR}(\mathfrak{X}_{K,\tilde{s}},\tilde{x})}(K) = \P(E_{\tilde{x}})^{K_\lambda}(K)  \subset \P(E_{\tilde{x}})^{K_\mathrm{colim}}(K)\]
and the proof is complete. 
 \end{proof}

\begin{proof}[Proof of Theorem \ref{theo: heslift}]
It follows from \cite[Proposition 2.2.7]{Ber96b} that the functors
\begin{align*}
\mathrm{Isoc}^\dagger(X/K) & \rightarrow \mathrm{MIC}(X,\mathfrak{X}/K) \\
\mathrm{Isoc}^\dagger(S/K) & \rightarrow \mathrm{MIC}(S,\mathfrak{S}/K) \\
\mathrm{Isoc}^\dagger(X_s/K) & \rightarrow \mathrm{MIC}(\mathfrak{X}_{\tilde{s}}/K)
\end{align*}
are all fully faithful with image stable by sub-quotients. In particular in the commutative diagram 
\[\xymatrix{ \pi_1^\mathrm{dR}(\mathfrak{X}_{K,\tilde{s}},\tilde{x})\ar[r] \ar@{->>}[d] &  \pi_1^\mathrm{colim}(]X[_\mathfrak{X},\tilde{x})\ar[r] \ar@{->>}[d] &  \pi_1^\mathrm{colim}(]S[_\mathfrak{S},\tilde{s}) \ar[r] \ar@{->>}[d] & 1  \\ \pi_1^\dagger(X_s,x) \ar[r] &  \pi_1^\dagger(X,x) \ar[r] &  \pi_1^\dagger(S,s) \ar[r] & 1  } \]
all the vertical maps are surjective. Since the top sequence is exact, and the bottom sequence is weakly exact, we may apply Lemma \ref{lemma: 6lem}.
 \end{proof}

\section{The case of a smooth family of curves}

While the liftability condition in Theorem \ref{theo: heslift} is extremely strong, it will always hold for a family of curves over a smooth affine base. That this is true is the key result of this section; we will then use this to deduce Theorem \ref{theo: main2} in relative dimension 1. In order to do so, we will need to show Theorem \ref{theo: main2} `localises on the base', which is the content of the following lemma. Throughout, let $f:X\rightarrow S$, $x\in X(k)$ and $s=f(x)$ be as in the statement of Theorem \ref{theo: main2}

\begin{lemma} \label{lemma: open} Let $U\subset S$ be an open subset containing $s$ and $f_U:X_U\rightarrow U$ the base change. If the homotopy sequence for $(f_U,x)$ is exact, then so is the homotopy sequence for $(f,x)$.
\end{lemma}

\begin{proof}
We consider the diagram
\[\xymatrix{   \pi_1^\dagger(X_s,x) \ar[r] \ar@{=}[d] &  \pi_1^\dagger(X_U,x)\ar[r] \ar@{->>}[d] &\pi_1^\dagger(U,s) \ar[r]\ar@{->>}[d] & 1 \\ \pi_1^\dagger(X_{s},x) \ar[r] & \pi_1^\dagger(X,x)\ar[r] & \pi_1^\dagger(S,s)\ar[r] & 1.  } \]
where the vertical arrows are surjective by \cite[Theorem 5.2.1, Proposition 5.3.1]{Ked07}. Since the bottom sequence is always weakly exact, we may apply Lemma \ref{lemma: 6lem} to conclude.
 \end{proof}

This enables us to prove exactness of the homotopy sequence for curves.

\begin{theorem} \label{theo: hescurves} Therem \ref{theo: main2} is true if $f:X\rightarrow S$ has relative dimension $1$.
\end{theorem}

\begin{proof}
Let $g$ denote the genus of $f$. If $g=0$ then we are done by the proof of Lemma \ref{lemma: projbun} - to make the argument work, we only need a rational point on the fibre over $s$, which we have by assumption. We will treat the case $g\geq 2$ in detail, and then point out where the argument needs to be modified to work for $g=1$.

By Proposition \ref{lemma: open} we are free to replace $S$ by any open sub-scheme containing $s$, in particular we may assume that $S=\spec{A_0}$ is affine, and $f:X\rightarrow S$ can be tri-canonically embedded in $\P^{5g-6}_S$. If we let $H_g^0$ denote the moduli scheme (over $\Z$) of such tri-canonically embedded curves, and $Z_g^0\rightarrow H_g^0$ the universal curve, then we obtain a Cartesian diagram of schemes
\[ \xymatrix{ X\ar[r]\ar[d] & Z_g^0 \ar[d] \\ S\ar[r] & H_g^0. } \]
After possibly shrinking $S$ further we may assume that there exists an open affine sub-scheme $V\subset H_g^0$ through which $S\rightarrow H_g^0$ factors. 

Now by \cite[Th\'{e}or\`{e}me 6]{Elk73} we may lift $A_0$ to a smooth $\ca{V}$-algebra $A$, let $A^h$ denote the $\varpi$-adic Henselisation of $A$. Since $V$ is smooth and affine over $\Z$ \cite[Corollary 1.7]{DM69} we may apply \cite[Th\'{e}or\`{e}me 2]{Ray72} to deduce that the given morphism $S\rightarrow V$ lifts to a morphism
\[ \spec{A^h}\rightarrow V. \]
Since $V$ is of finite type over $\Z$, it follows that after possibly passing to some \'{e}tale $A$-algebra with the same special fibre, we may assume that the family of curves $X\rightarrow S$ lifts to a family $\ca{X}\rightarrow \ca{S}=\spec{A}$ over $A$, which is moreover a closed sub-scheme of $\P^{5g-6}_\ca{S}$. Now choose a projective embedding
\[ \ca{S} \hookrightarrow \A^N_\ca{V} \hookrightarrow \P^N_\ca{V}, \]
let $\overline{\ca{S}}$ be the closure of $\ca{S}$ inside $\P^N_\ca{V}$, and let $\overline{\ca{X}}$ be the closure of $\ca{X}$ inside $\P^{5g-6}_{\overline{\ca{S}}}$. Setting $\mathfrak{X}=\widehat{\overline{\ca{X}}}$ and $\mathfrak{S}=\widehat{\overline{\ca{S}}}$ we find ourselves in the situation of Theorem \ref{theo: heslift}.

When $g=1$ we can argue as follows. First of all, we consider the base change $X\times_S X\rightarrow X$ of $f$ by itself, equipped with the rational point $(x,x)$. Then we have a commutative diagram
\[ \xymatrix{ \pi_1^\dagger(X_s,x)\ar[r] \ar@{=}[d] & \pi_1^\dagger(X\times_S X,(x,x)) \ar[r] \ar@{->>}[d] & \pi_1^\dagger(X,x) \ar@{->>}[d] \ar[r] & 1 \\ \pi_1^\dagger(X_s,x) \ar[r] & \pi_1^\dagger(X,x) \ar[r] & \pi_1^\dagger(S,s) \ar[r] & 1  } \]
where the surjectivity of the vertical arrows follows from Corollary \ref{cor: weakhes}. If the top sequence is exact, then again the bottom sequence is exact by Lemma \ref{lemma: 6lem}. In particular, after replacing $X\rightarrow S$ by $X\times_S X\rightarrow X$, we may assume that $f:X\rightarrow S$ admits a section, i.e. is an elliptic curve.

Hence after possibly localising on $S$ and again using Lemma \ref{lemma: open}, we may assume that we have a smooth Weierstrass model $X\hookrightarrow \P^2_S$ of $X$. We now replace the scheme $H^0_g$ in the previous argument with the smooth moduli scheme over $\Z$ parametrising Weierstrass models of elliptic curves.
\end{proof}

\section{Cutting a smooth projective morphism by curves}

Using Theorem \ref{theo: weaklef}, we can now finally complete the proof of Theorem \ref{theo: main2} by reducing to the case of a family of smooth projective curves, and hence to Theorem \ref{theo: hescurves}.

So suppose that we are in the situation of Theorem \ref{theo: main2}, let $d$ be the relative dimension of $X/S$, and assume that $d\geq 2$. Assume moreover that $S$ is quasi-projective over $k$, there therefore exists a global closed immersion $X\hookrightarrow \P^n_S$. Let $\widetilde{S}$ denote the dual projective space $\check{\P}^n_S$, and define
\[ \widetilde{X} \subset X\times_S \widetilde{S} \]
to be the sub-scheme of pairs $(x,H)$ such that $x\in H$. We therefore have a diagram
\[ \xymatrix{ \widetilde{X} \ar[r]\ar[d]_{\tilde f} & X \ar[d]^f \\ \widetilde{S} \ar[r] & S  } \]
such that $\widetilde{f}$ is projective, with generic fibre smooth, projective and geometrically connected of dimension $d-1$. Choose an open sub-scheme $U\subset \widetilde{S}$, surjective over $S$, such that the pull-back $\tilde{f}_U:\widetilde{X}_U\rightarrow U$ is smooth with geometrically connected fibres. After possibly making a finite extension of $k$, we may moreover assume that there exists some $k$-rational point $u\in U$ lifting $s\in S(k)$, and some $\tilde{x}\in \widetilde{X}_U(k)$ such that $\tilde{f}_U(\tilde{x})=u$.

In particular, combining Lemma \ref{lemma: projbun} with \cite[Theorem 5.2.1, Proposition 5.3.1]{Ked07} we have a commutative diagram
\[\xymatrix{   \pi_1^\dagger(\widetilde{X}_{u},\tilde x) \ar[r] \ar@{=}[d] &  \pi_1^\dagger(\widetilde{X}_U,\tilde x)\ar[r] \ar@{->>}[d] &\pi_1^\dagger(U,u) \ar[r]\ar@{->>}[d] & 1 \\
\pi_1^\dagger(\widetilde{X}_{u},\tilde x) \ar[r] \ar[d] &  \pi_1^\dagger(\widetilde{X},\tilde x)\ar[r] \ar[d]^{\cong} &\pi_1^\dagger(\widetilde S,u) \ar[r]\ar[d]^{\cong} & 1 \\  \pi_1^\dagger(X_{s},x) \ar[r] & \pi_1^\dagger(X,x)\ar[r] & \pi_1^\dagger(S,s)\ar[r] & 1.  } \]

\begin{proposition} \label{prop: redcurves} If the homotopy sequence of $(\tilde{f}_U,\tilde x)$ is exact, then so is that of $(f,x)$.
\end{proposition}

\begin{proof} We first claim that under the hypothesis of the proposition, the sequence 
\[ \pi_1^\dagger(\widetilde{X}_{u},\tilde x) \rightarrow  \pi_1^\dagger(\widetilde{X},\tilde x) \rightarrow \pi_1^\dagger(\widetilde S,u)\rightarrow 1\]
satisfies the conditions of Theorem \ref{theo: tann} (note that this does \emph{not} follow from Corollary \ref{cor: weakhes}). Indeed, surjectivity of $\pi_1^\dagger(\widetilde{X},\tilde x) \rightarrow \pi_1^\dagger(\widetilde S,u)$ follows from that of $\pi_1^\dagger(X,x) \rightarrow \pi_1^\dagger(S,s)$, and one half of (1) is clear.

For the other half of (1), suppose that $E\in \mathrm{Isoc}^\dagger(\widetilde{X}/K)$ is such that $E|_{\widetilde{X}_u}$ is trivial. We may assume by Lemma \ref{lemma: projbun} that $E$ comes from an object $E'$ of $\mathrm{Isoc}^\dagger(X/K)$. Applying Theorem \ref{theo: weaklef} to the map $\widetilde{X}_u \rightarrow X_s$ we can see that in fact $E'\mid_{X_s}$ is trivial, and so by weak exactness of 
\[ \pi_1^\dagger(X_s,x) \rightarrow  \pi_1^\dagger(X,x) \rightarrow \pi_1^\dagger(S,s)\rightarrow 1 \]
we know that $E'\cong f^*(F')$ for some $F'\in \mathrm{Isoc}^\dagger(S/K)$. Hence $E\cong \tilde{f}^*(F)$ for some $F\in \mathrm{Isoc}^\dagger(\widetilde S/K)$ as required. To prove (2) we note that if the top sequence is exact then the image of 
\[ \pi_1^\dagger(\widetilde{X}_{u},\tilde x) \rightarrow  \pi_1^\dagger(\widetilde{X},\tilde x) \]
is a normal subgroup, and hence for any representation $V$ of $\pi_1^\dagger(\widetilde{X},x)$, we know that $V^{\pi_1^\dagger(\widetilde{X}_u,\tilde x)}$ must be stable under $\pi_1^\dagger(\widetilde{X},\tilde x)$. 

Hence the top sequence is weakly exact. Since we have already noted that the image of $\pi_1^\dagger(\widetilde{X}_{u},\tilde x) \rightarrow  \pi_1^\dagger(\widetilde{X},\tilde x)$
is a normal subgroup, it is in fact exact. The exactness of
\[ \pi_1^\dagger(X_s,x) \rightarrow  \pi_1^\dagger(X,x) \rightarrow \pi_1^\dagger(S,s)\rightarrow 1\]
now follows from a simple diagram chase.
 \end{proof}

We can now complete the proof of Theorem \ref{theo: main2}.

\begin{proof}[Proof of Theorem \ref{theo: main2}] By Lemma \ref{lemma: open} we may assume that $S$ is quasi-projective; we will induct on the relative dimension $d$. The inductive step is taken care of by Proposition \ref{prop: redcurves}, and the base case of $d=1$ is Theorem \ref{theo: hescurves}.

 \end{proof}

\section{Applications} \label{sec: app}

In this final section we deduce a couple of corollaries of Theorem \ref{theo: main2}.

\begin{theorem} \label{theo: lefschetz} Let $X\subset \P^n_k$ be smooth, projective and geometrically connected. Assume that $\dim X\geq 2$ and let $Y=H\cap X$ be a smooth hyperplane section. Let $y\in Y(k)$. Then the map
\[ \pi_1^\dagger(Y,y)\rightarrow \pi_1^\dagger(X,y)\]
is surjective. 
\end{theorem}

\begin{proof} We simply copy the proof of Theorem \ref{theo: weaklef}, replacing all instances of Corollary \ref{cor: weakhes} with Theorem \ref{theo: main2}.
\end{proof}

We can also show a version of the Kunneth formula, when one of the varieties is projective.

\begin{proposition} Let $X,Y$ be smooth and geometrically connected over $k$, $x\in X(k)$, $y\in Y(k)$. If either $X$ or $Y$ is projective, then the natural map
\[ \pi_1^\dagger(X\times_k Y,(x,y))\rightarrow \pi_1^\dagger(X,x)\times _K \pi_1^\dagger(Y,y) \]
is an isomorphism.
\end{proposition}

\begin{proof} Assume that $Y$ is projective. Then by Theorem \ref{theo: main2} the sequence 
\[ \pi_1^\dagger(Y,y)\rightarrow  \pi_1^\dagger(X\times_k Y,(x,y)) \rightarrow \pi_1^\dagger(X,x)\rightarrow 1 \]
is exact, and the first map is split by the projection $\pi_1^\dagger(X\times_k Y,(x,y))\rightarrow \pi_1^\dagger(Y,y)$.
\end{proof}

We can also use Theorem \ref{theo: main2} to compare the (over)convergent fundamental group $\pi_1^\dagger(X)$ with the \'etale one $\pi_1^\et(X)$.

\begin{assumptionu} We will assume for the remainder of this section that the ground field $k$ is algebraically closed.
\end{assumptionu}

Let $X$ be a smooth, projective, connected $k$-variety, and $x\in X(k)$. Then there is a functor
\[ \mathrm{\acute{E}t}(X)\rightarrow \mathrm{Isoc}^\dagger(X/K) \]
from the category of finite \'etale covers of $X$ to (over)convergent isocrystals which sends $f:Y\rightarrow X$ to $f_*\ca{O}_{Y/K}^\dagger\in \mathrm{Isoc}^\dagger(X/K)$. This gives rise to a homomorphism of pro-algebraic groups
\[ \pi_1^\dagger(X,x)\rightarrow \pi_1^\et(X,x) \]
and hence to a homomorphism of pro-finite groups
\[  \pi_0(\pi_1^\dagger(X,x))\rightarrow \pi_1^\et(X,x)\]
where the former is the component group of $\pi_1^\dagger(X,x)$. The following strengthens a result of Crew \cite[Proposition 4.4]{Cre92}.

\begin{theorem} \label{theo: pi0}
The map $\pi_0(\pi_1^\dagger(X,x))\rightarrow \pi_1^\et(X,x)$ is an isomorphism.
\end{theorem}

We can translate this into Tannakian terms as follows. Let $E\in \mathrm{Isoc}^\dagger(X/K)$, and consider the associated monodromy representation 
\[ \pi_1^\dagger(X,x)\rightarrow \mathrm{GL}(E_x). \]
By definition the image of this homomorphism is the monodromy group $\mathrm{DGal}_x(E)$ of $E$. Then Theorem \ref{theo: pi0} is equivalent to the following.

\begin{theorem} \label{theo: pi02} If $E\in \mathrm{Isoc}^\dagger(X/K)$ has \emph{finite} monodromy group, then it is trivialised by some finite \'etale cover $f:Y\rightarrow X$
\end{theorem}

\begin{remark} Note that the condition of having finite monodromy group is independent of the choice of base point $x$.
\end{remark}

The following lemma will play an important role in the proof of Theorem \ref{theo: pi0}. 

\begin{lemma} \label{lemma: finjex} Let $f:Y\rightarrow X$ be a morphism of connected $k$-varieties, with $Y$ a smooth curve. Let $y\in Y(k)$, and set $x=f(y)$. Assume that the induced homomorphism
\[ \pi_1^\dagger(Y,y)\rightarrow \pi_1^\dagger(X,x) \]
is surjective, and let $E\in \mathrm{Isoc}^\dagger(X/K)$ be simple. If there exists $E'\in F\text{-}\mathrm{Isoc}^\dagger(Y/K)$ and an injection of isocrystals $f^*E\hookrightarrow E'$, then there exists $E'' \in F\text{-}\mathrm{Isoc}^\dagger(X/K)$ and an injection of isocrystals $E\hookrightarrow E''$.
\end{lemma}

\begin{proof} Write $f^*(-)=(-)|_Y$. For all $n\geq 0$ we can view the Frobenius pullback $(F^n)^*E|_Y$ as lying inside $E'$ via the Frobenius $\varphi^n$ on $E'$. Let $E'_0\subset E'$ denote the sum of all the $(F^n)^*E|_Y$, this is therefore stable by $\varphi$, and we can thus view it as an $F$-isocrystal. Since $E$ is simple, surjectivity of the map
\[ \pi_1^\dagger(Y,y)\rightarrow \pi_1^\dagger(X,x)\]
implies that $E|_Y$ is simple. Moreover, since $Y$ is a smooth curve, we know that there exists a good compactification $\overline{Y}$ which lifts to a smooth and proper curve over $\mathcal{V}$. Thus we may apply \cite[Th\'eor\`eme 4.4.5]{Ber96a} together with \cite[Th\'eor\`eme 2.3.6]{Ber00} to deduce that the Frobenius pull-back functor
\[ F^*:\isoc{Y}\rightarrow \isoc{Y} \]
is an equivalence of categories. In particular each $(F^n)^*E|_Y$ is simple. Hence $E'_0$ is semi-simple, and moreover we can write
\[ E'_0\cong \bigoplus_i (F^{n_i})^*E|_Y \cong \left(\bigoplus_i (F^{n_i})^*E\right)|_Y \]
for some $n_i$. Hence $E'_0$ extends to some $E''$ on $X$, and since 
\[ \pi_1^\dagger(Y,y)\rightarrow \pi_1^\dagger(X,x)\]
is surjective both the induced Frobenius on $E'_0=E''|_Y$, as well as the injection $E|_Y\hookrightarrow E''|_Y$ also extend to $X$. \end{proof}

We can now begin the proof of Theorem \ref{theo: pi0} in earnest. Our first step is to show a slightly weaker version of Theorem \ref{theo: pi02} when $X$ is a curve.

\begin{lemma} \label{lemma: curvtriv} Let $C$ be a smooth, projective, connected curve over $k$, and let $E\in \mathrm{Isoc}^\dagger(C/K)$ be simple with finite monodromy group. Then there exists a smooth, projective, connected curve $D$ and a non-constant, separable morphism $f:D\rightarrow C$ such that $f^*E$ is trivial.
\end{lemma}

\begin{proof} Choose a lift $\ca{C}$ of $C$ to $\ca{V}$, with generic fibre $\ca{C}_K$ a smooth, projective, geometrically connected curve over $K$. Then $E$ corresponds to a module with overconvergent connection on $\ca{C}_K$. Choose a point $\tilde{c}$ of $\ca{C}_K$ specialising to some $c\in C(k)$. Combining the defintion of $\mathrm{Isoc}^\dagger(C/K)$ with GAGA, we can see that $\mathrm{Isoc}^\dagger(C/K)$ is a full subcategory of the category $\mathrm{MIC}(\ca{C}_K/K)$ of modules with integrable connection on the algebraic curve $\ca{C}_K$, and this subcategory is stable under taking subquotients. We therefore get a surjective map $\pi_1^\dR(\ca{C}_K,\tilde{c})\rightarrow \pi_1^\dagger(C,c)$.

Hence $E$ also has finite monodromy when considered as a representation of $\pi_1^\dR(\ca{C}_K,\tilde{c})$. If we now choose a finitely generated subfield $K_0\subset K$ over which $\ca{C}_K$, $\tilde{c}$, and $E$ are all defined, and then embed this field into $\C$, we have an isomorphism between $\pi_1^\dR(\ca{C}_{\C},\tilde{c})$ and the $\C$-valued pro-algebraic completion of the abstract group $\pi_1(\ca{C}(\C),\tilde{c})$. Thus $E_{\C}$ has finite monodromy as a representation of $\pi_1(\ca{C}(\C),\tilde{c})$, and so is trivialised by a finite \'etale cover of $\ca{C}_{\C}$. Descending down to $K_0$, and then re-ascending to $K$, we can see that after possibly replacing $K$ by a finite extension (this does not change the problem) there exists a finite, \'etale, geometrically connected cover $\ca{C}'_K\rightarrow \ca{C}_K$ trivialising $E$ as a module with integrable connection. Let $\ca{C}'$ denote the normalisation of $\ca{C}$ inside the function field extension $K(\ca{C}_K)\rightarrow K(\ca{C}'_K)$.

By de Jong's theorem on alterations \cite[Theorem 8.2]{dJ96} we can find (after possibly further increasing $K$) some alteration $\ca{C}''\rightarrow \ca{C}'$ with strictly semistable reduction, let $C''$ denote the special fibre. Then the map from the smooth locus of $C''$ to $C$ is dominant, so we may choose some connected component $C'''$ of $\mathrm{sm}(C'')$ which is non-constant over $C$. Then the formal completion $\widehat{\ca{C}}''$ is smooth over $\ca{V}$ in a neighbourhood of $C'''$, and so the pull-back of isocrystals along $C'''\rightarrow C$ can be identified with the pull-back of modules with integrable connection along
\[  ]C'''[_{\widehat{\ca{C}}''} \subset {\ca{C}_K''}^\mathrm{an}\rightarrow{\ca{C}_K'}^\mathrm{an} \rightarrow \ca{C}_K^\mathrm{an}. \]
Thus $E$ becomes trivial after pulling back by $C'''\rightarrow C$. Let $D'$ denote the smooth compactification of $C'''$, applying \cite[Theorem 5.2.1]{Ked07} we therefore know that $E$ becomes trivial after pulling back by $D'\rightarrow C$. Finally, factor $D'\rightarrow D\rightarrow C$ into a totally inseparable morphism $D'\rightarrow D$ followed by a separable morphism $D\rightarrow C$. Since $E$ is trivial when pulled back to $D'$, it follows from \cite[Corollary 4.10]{Ogu84} that it is trivial when pulled back to $D$. \end{proof}

We can now prove Theorem \ref{theo: pi02}, and hence Theorem \ref{theo: pi0}.

\begin{proof}[Proof of Theorem \ref{theo: pi02}] 
Let $E\in \mathrm{Isoc}^\dagger(X/K)$ have finite monodromy group. Since finite groups are reductive, $E$ is therefore semi-simple, and by working one simple factor at a time we may assume that $E$ is simple. Choose an iterated hyperplane section $C\rightarrow X$ which is a smooth, projective, connected curve over $k$. By Theorem \ref{theo: lefschetz} we know that $E|_C$ is simple with finite monodromy, and therefore by Lemma \ref{lemma: curvtriv} we can find some finite separable morphism $D\rightarrow C$ such that $E|_D$ is trivial.

Let $D_U\rightarrow U$ be the \'etale locus of $D\rightarrow C$. Then since $E|_{D_U}$ is trivial, we know by taking the pushforward along $D_U\rightarrow U$ that we can find some $E'\in F\text{-}\mathrm{Isoc}^\dagger(U/K)$ and an injection of isocrystals $E|_U\hookrightarrow E'$. Using Theorem \ref{theo: lefschetz} together with \cite[Theorem 5.2.1, Proposition 5.3.1]{Ked07} we may apply Lemma \ref{lemma: finjex} above to the map $U\rightarrow X$. We can therefore find some $E''\in F\text{-}\mathrm{Isoc}^\dagger(X/K)$ and an injection of isocrystals $E\hookrightarrow E''$. We can now apply \cite[Proposition 4.3]{Cre92} to conclude. (Note the remark following the proof of this result that you do not need to assume $X$ is a curve for the implication (iv)$\Rightarrow$(ii), which is the one we are interested in.)

\end{proof}

\bibliographystyle{../../Templates/bibsty}
\bibliography{../../Templates/lib.bib}

\providecommand{\bysame}{\leavevmode\hbox to3em{\hrulefill}\thinspace}
\providecommand{\MR}{\relax\ifhmode\unskip\space\fi MR }
\providecommand{\MRhref}[2]{%
  \href{http://www.ams.org/mathscinet-getitem?mr=#1}{#2}
}
\providecommand{\href}[2]{#2}
\begin{thebibliography}{DMOS82}

\bibitem[Ber74]{Ber74}
P.~Berthelot, \emph{Cohomologie cristalline des sch\'emas de caract\'eristique
  {$p>0$}}, Lecture Notes in Mathematics, vol. 407, Springer, 1974,
  \url{https://dx.doi.org/10.1007/BFb0068636}.

\bibitem[Ber96a]{Ber96b}
\bysame, \emph{Cohomologie rigide et cohomologie rigide {\`a} supports propres,
  premi{\`e}re partie}, preprint (1996),
  \url{http://perso.univ-rennes1.fr/pierre.berthelot/publis/Cohomologie_Rigide_I.pdf}.

\bibitem[Ber96b]{Ber96a}
\bysame, \emph{{${\mathcal D}$}-modules arithm\'etiques. {I}. {O}p\'erateurs
  diff\'erentiels de niveau fini}, Ann. Sci. \'Ecole Norm. Sup. (4) \textbf{29}
  (1996), no.~2, 185--272,
  \url{http://www.numdam.org/item?id=ASENS_1996_4_29_2_185_0}.

\bibitem[Ber00]{Ber00}
\bysame, \emph{{$\mathcal D$}-modules arithm\'etiques. {II}. {D}escente par
  {F}robenius}, M\'em. Soc. Math. Fr. (N.S.) (2000), no.~81.

\bibitem[Ber02]{Ber02}
\bysame, \emph{Introduction {\`a} la th{\'e}orie arithm{\'e}tique des
  {$\mathscr{D}$}-modules}, Cohomologies {$p$}-adiques et applications
  arithm{\'e}tiques, II, no. 279, Asterisque, 2002, pp.~1--80.

\bibitem[Car06]{Car06b}
D.~Caro, \emph{Fonctions {L} associ{\'e}s aux {$\mathcal{D}$}-modules
  arithm\'{e}tiques. {C}as des courbes}, Comp. Math. \textbf{142} (2006),
  169--206, \url{http://dx.doi.org/10.1112/S0010437X05001880}.

\bibitem[Car09]{Car09a}
\bysame, \emph{$\mathscr{D}$-modules arithm\'{e}tiques associ\'{e}s aux
  isocristaux surconvergents. {C}as lisse}, Bulletin de la S.M.F. \textbf{137}
  (2009), no.~4, 453--543.

\bibitem[Car11]{Car11}
\bysame, \emph{Pleine fid\'{e}lit\'{e} sans structure de {F}robenius et
  isocristaux partiellement surconvergents}, Math. Ann. \textbf{349} (2011),
  no.~4, 747--805, \url{http://dx.doi.org/10.1007/s00208-010-0539-x}.

\bibitem[Car15]{Car15a}
\bysame, \emph{Sur la pr\'eservation de la surconvergence par l'image directe
  d'un morphisme propre et lisse}, Ann. Sci. \'Ec. Norm. Sup\'er. (4)
  \textbf{48} (2015), no.~1, 131--169,
  \url{http://smf4.emath.fr/Publications/AnnalesENS/4_48/html/ens_ann-sc_48_131-169.php}.

\bibitem[CDPS]{CDPS}
B.~Chiarellotto, V.~Di~Proietto, and A.~Shiho, \emph{Comparison of relatively
  unipotent log de {R}ham fundamental groups}, to appear in Memoires of the
  A.M.S.

\bibitem[Con06]{Con06}
B.~Conrad, \emph{Relative ampleness in rigid geometry}, Ann. Inst. Fourier
  (Grenoble) \textbf{56} (2006), no.~4, 1049--1126,
  \url{http://aif.cedram.org/item?id=AIF_2006__56_4_1049_0}.

\bibitem[Cre92]{Cre92}
R.~Crew, \emph{{$F$}-isocrystals and their monodromy groups}, Ann. Sci. Ecole.
  Norm. Sup. \textbf{25} (1992), no.~4, 429--464,
  \url{http://www.numdam.org/item?id=ASENS_1992_4_25_4_429_0}.

\bibitem[dJ96]{dJ96}
A.~J. de~Jong, \emph{Smoothness, semi-stability and alterations}, Inst. Hautes
  \'Etudes Sci. Publ. Math. \textbf{83} (1996), no.~1, 51--93,
  \url{http://www.numdam.org/item?id=PMIHES_1996__83__51_0}.

\bibitem[DM69]{DM69}
P.~Deligne and D.~Mumford, \emph{The irreducibility of the space of curves of
  given genus}, Inst. Hautes \'Etudes Sci. Publ. Math. (1969), no.~36, 75--109,
  \url{http://www.numdam.org/item?id=PMIHES_1969__36__75_0}.

\bibitem[DMOS82]{DMOS82}
P.~Deligne, J.~S. Milne, A.~Ogus, and K.~Shih, \emph{Hodge cycles, motives, and
  {S}himura varieties}, Lecture Notes in Mathematics, vol. 900,
  Springer-Verlag, Berlin-New York, 1982,
  \url{https://dx.doi.org/10.1007/978-3-540-38955-2}.

\bibitem[DPS18]{DPS18}
V.~Di~Proietto and A.~Shiho, \emph{On the homotopy exact sequence for log
  algebraic fundamental group}, Documenta Math. \textbf{23} (2018), 543--597,
  \url{https://doi.org/10.25537/dm.2018v23.543-597 }.

\bibitem[dS15]{dS15}
J.~P. dos Santos, \emph{The homotopy exact sequence for the fundamental group
  scheme and infinitesimal equivalence relations}, Alg. Geom. \textbf{2}
  (2015), no.~5, 535--590, \url{http://dx.doi.org/10.14231/AG-2015-024}.

\bibitem[EH06]{EH06}
H.~Esnault and P.~H. Hai, \emph{The {G}auss-{M}anin connection and {T}annaka
  duality}, Int. Math. Res. Not. (2006), Art. ID 93978, 35,
  \url{http://dx.doi.org/10.1155/IMRN/2006/93978}.

\bibitem[EHS08]{EHS08}
H.~Esnault, P.~H. Hai, and X.~Sun, \emph{On {N}ori's fundamental group scheme},
  Geometry and dynamics of groups and spaces, Progr. Math., vol. 265,
  Birkh\"auser, Basel, 2008,
  \url{http://dx.doi.org/10.1007/978-3-7643-8608-5_8}, pp.~377--398.

\bibitem[Elk73]{Elk73}
R.~Elkik, \emph{Solutions d'\'equations \`a coefficients dans un anneau
  hens\'elien}, Ann. Sci. \'Ecole Norm. Sup. (4) \textbf{6} (1973), 553--603
  (1974), \url{http://www.numdam.org/item?id=ASENS_1973_4_6_4_553_0}.

\bibitem[FK18]{FK18}
K.~Fujiwara and F.~Kato, \emph{Foundations of {R}igid {G}eometry. {I}}, EMS
  Monographs in Mathematics, European Mathematical Society (EMS), Z\"{u}rich,
  2018, \url{http://www.ems-ph.org/doi/10.4171/135}.

\bibitem[Fri73]{Fri73}
E.~M. Friedlander, \emph{The \'etale homotopy theory of a geometric fibration},
  Manuscripta Math. \textbf{10} (1973), 209--244,
  \url{http://dx.doi.org/10.1007/BF01332767}.

\bibitem[Gro61]{Gro60}
A.~Grothendieck, \emph{Technique de descente et th\'eor\`emes d'existence en
  g\'eom\'etrie alg\'ebrique {IV}: les sch\'emas de {H}ilbert}, S\'eminaire
  {B}ourbaki, {V}ol.\ 6, Soc. Math. France, Paris, 1960-61,
  \url{http://www.numdam.org/item?id=SB_1960-1961__6__249_0 }, pp.~Exp.\ No.\
  221, 249--276.

\bibitem[Gro61]{EGA3.I}
A.~Grothendieck, \emph{{\'E}l\'ements de g\'eom\'etrie alg\'ebrique. {III}.
  \'{E}tude cohomologique des faisceaux coh\'e{}érents. {I}}, Inst. Hautes
  \'Etudes Sci. Publ. Math. (1961), no.~11, 5--167, R\'edig\'es avec la
  collaboration de J. Dieudonn\'e,
  \url{http://www.numdam.org/item?id=PMIHES_1964__20__259_0}.

\bibitem[Gro66]{EGA4.3}
\bysame, \emph{{\'E}l\'ements de g\'eom\'etrie alg\'ebrique. {IV}. \'{E}tude
  locale des sch\'emas et des morphismes de sch\'emas {III}}, Inst. Hautes
  \'Etudes Sci. Publ. Math. (1966), no.~28, 255, R\'edig\'es avec la
  collaboration de J. Dieudonn\'e,
  \url{http://www.numdam.org/item?id=PMIHES_1966__28__5_0}.

\bibitem[Hub96]{Hub96}
R.~Huber, \emph{\'{E}tale cohomology of rigid analytic varieties and adic
  spaces}, Aspects of Mathematics, E30, Friedr. Vieweg \& Sohn, Braunschweig,
  1996, \url{http://dx.doi.org/10.1007/978-3-663-09991-8}.

\bibitem[Kat70]{Kat70}
N.~Katz, \emph{Nilpotent connections and the monodromy theorem: Applications of
  a result of {T}urrittin}, Publ. Math. I.H.E.S. \textbf{39} (1970), 175--232,
  \url{http://eudml.org/doc/103909}.

\bibitem[Ked07]{Ked07}
K.~S. Kedlaya, \emph{Semistable reduction for overconvergent {$F$}-isocrystals
  {I}: {U}nipotence and logarithmic extensions}, Comp. Math. \textbf{143}
  (2007), 1164--1212, \url{http://dx.doi.org/10.1112/S0010437X07002886}.

\bibitem[KO68]{KO68}
N.~Katz and T.~Oda, \emph{On the differentiation of de {R}ham cohomology
  classes with respect to parameters}, J. Math. Kyoto Univ. \textbf{8} (1968),
  no.~2, 199--213, \url{http://dx.doi.org/10.1215/kjm/1250524135}.

\bibitem[Laz15]{Laz15}
C.~Lazda, \emph{Relative fundamental groups and rational points}, Rend. Sem.
  Mat. Univ. Padova \textbf{134} (2015), 1--45,
  \url{http://dx.doi.org/10.4171/RSMUP/134-1}.

\bibitem[LS07]{LS07}
B.~Le~Stum, \emph{Rigid cohomology}, Cambridge Tracts in Mathematics, vol. 172,
  Cambridge University Press, Cambridge, 2007,
  \url{http://dx.doi.org/10.1017/CBO9780511543128}.

\bibitem[Ogu84]{Ogu84}
A.~Ogus, \emph{{$F$}-isocrystals and de {R}ham cohomology {II} - {C}onvergent
  isocrystals}, Duke Math. J. \textbf{51} (1984), no.~4, 765--850,
  \url{http://dx.doi.org/10.1215/S0012-7094-84-05136-6}.

\bibitem[Ray72]{Ray72}
M.~Raynaud, \emph{Anneaux hens\'eliens et approximations}, 9 pp. Publ. S\'em.
  Math. Univ. Rennes, Ann\'ee 1972,
  \url{http://www.numdam.org/item?id=PSMIR_1972___4_A13_0}.

\bibitem[To{\"{e}}00]{Toe00}
B.~To{\"{e}}n, \emph{Dualit\'{e} de {T}annaka sup\'{e}rieure {I}: {S}tructures
  monoidales}, preprint (2000),
  \url{https://perso.math.univ-toulouse.fr/btoen/unpublished-texts/}.

\bibitem[Zha14]{Zha14}
L.~Zhang, \emph{The {H}omotopy {S}equence of the {A}lgebraic {F}undamental
  {G}roup}, Int. Math. Res. Notices \textbf{22} (2014), 6155--6174,
  \url{https://doi.org/10.1093/imrn/rnt163}.

\end{thebibliography}

\end{document}